\documentclass[11pt, letterpaper]{article}

\usepackage{amssymb,amsmath,amsfonts}
\usepackage{hyperref}
\usepackage{graphicx,color}
\usepackage{mathrsfs}		
\usepackage{amsthm}	
\usepackage[round]{natbib}
\usepackage[dvipsnames]{xcolor}
\usepackage[]{appendix}
\usepackage{tikz} 
\usepackage{subfigure}
\usetikzlibrary{patterns}
\usepackage{float}
\usepackage{pgfplots}
\usepackage{bm}

\usepackage{geometry}

\usepackage{enumitem}

\newcommand{\E}{{\mathbb E}}
\newcommand{\F}{{\mathbb F}}

\renewcommand{\P}{{\mathbb P}}
\newcommand{\Q}{{\mathbb Q}}
\newcommand{\C}{{\mathbb C}}
\newcommand{\R}{{\mathbb R}}
\renewcommand{\S}{{\mathbb S}}
\newcommand{\N}{{\mathbb N}}

\newcommand{\Fcal}{{\mathcal F}}
\newcommand{\Gcal}{{\mathcal G}}

\newcommand{\Kcal}{{\mathcal K}}

\newcommand{\Wcal}{{\mathcal W}}

\newcommand{\lkr}{L\'evy-Khintchine representation}

\newcommand{\Pol}{{\rm Pol}}
\newcommand{\im}{{\rm Im}}

\renewcommand{\colon}{{\ :\ }}
\DeclareMathOperator{\tr}{Tr}

\usepackage{dsfont}
\newcommand{\de}{{d}}
\newcommand{\e}{\varepsilon}
\newcommand{\supp}{{\textup{supp}}}
\newcommand{\Arg}{{\textup{Arg}}}

\newcommand{\Int}{{\textup{int}}}
\newcommand{\Tr}{{\textup{Tr}}}
\newcommand{\bk}{{\bm k}}
\newcommand{\bn}{{\bm n}}

\newcommand{\fdot}{{\,\cdot\,}}
\newcommand{\kk}{10}

\newtheorem{theorem}{Theorem}

\newtheorem{lemma}[theorem]{Lemma}

\theoremstyle{definition}
\newtheorem{definition}[theorem]{Definition}
\newtheorem{remark}[theorem]{Remark}

\newtheorem{assumption}{Assumption}

\numberwithin{equation}{section}
\numberwithin{theorem}{section}
\newtheorem{examplex}[theorem]{Example}
\newenvironment{example}
  {\pushQED{\qed}\examplex}
  {\popQED\endexamplex}

\begin{document}

\title{Polynomial jump-diffusions on the unit simplex\footnote{Acknowledgment: The authors wish to thank Thomas Bernhardt and Thomas Krabichler for useful comments and fruitful discussions. Martin Larsson and Sara Svaluto-Ferro gratefully acknowledges financial support by the Swiss National Science Foundation (SNF) under grant 205121\textunderscore163425.}}

\author{Christa Cuchiero\thanks{Vienna University, Oskar-Morgenstern-Platz 1, A-1090 Wien, Austria, christa.cuchiero@univie.ac.at}
\and Martin Larsson\thanks{Department of Mathematics, ETH Zurich, R\"amistrasse 101, CH-8092, Zurich, Switzerland, martin.larsson@math.ethz.ch.}
\and Sara Svaluto-Ferro\thanks{Department of Mathematics, ETH Zurich, R\"amistrasse 101, CH-8092, Zurich, Switzerland, sara.svaluto@math.ethz.ch.}}

\maketitle

\begin{abstract}
Polynomial jump-diffusions constitute a class of tractable stochastic models with wide applicability in areas such as mathematical finance and population genetics. We provide a full parameterization of polynomial jump-diffusions on the unit simplex under natural structural hypotheses on the jumps. As a stepping stone, we characterize well-posedness of the martingale problem for polynomial operators on general compact state spaces.
\end{abstract}




\section{Introduction}

Tractable families of Markov processes on the \emph{unit simplex}, 
featuring both \emph{diffusion and jump} components, are challenging to construct, 
yet play an important role in a host of applications. 
These include population genetics \citep{E:11,EM:13}, dynamic modeling of probabilities \citep{GJ:06}, 
and mathematical finance, in particular stochastic portfolio theory \citep{F:02,FK:05}. 
The present article addresses this challenge by specifying Markovian jump-diffusions on the unit simplex 
that are {\em polynomial}, meaning that the (extended) generator 
maps any polynomial to a polynomial of the same or lower degree.

Polynomial processes were introduced in \cite{Cuchiero/etal:2012}, see also \cite{Filipovic/Larsson:2016}, and are inherently tractable. Indeed, any polynomial jump-diffusion
\begin{enumerate}
\item is an It\^o semimartingale, meaning that its semimartingale characteristics are absolutely continuous with respect to Lebesgue measure. This justifies the name jump-diffusion in the sense of \citet[Chapter III.2]{Jacod/Shiryaev:2003}; 
\item admits explicit expressions for all moments in terms of  matrix exponentials. 
\end{enumerate}
The computational advantage associated with the second property has been exploited in a large variety of problems. In particular, applications in mathematical finance include interest rates \citep{Delbaen/Shirakawa:2002,FLT:16}, credit risk \citep{Ackerer/Filipovic:2016}, stochastic volatility models \citep{ack_fil_pul_16}, stochastic portfolio theory \citep{C:16,CGGPST:16}, life insurance liabilities \citep{BZ:16}, and variance swaps \citep{Filipovic/Gourier/Mancini:2015}.

In addition, polynomial jump-diffusions are highly flexible in that they allow for a wide range of state spaces
-- the unit simplex being one of them -- and a multitude of possible jump and diffusion phenomena. This stands in contrast to the thoroughly studied and frequently used sub-class of affine processes. Any affine jump-diffusion that admits moments of all orders is polynomial, but there are many polynomial jump-diffusions that are not affine. In particular, an affine process on a compact and connected state space is necessarily deterministic; see \cite{KL:16}. Thus our interest in the unit simplex forces us to look beyond the affine class.

Polynomial diffusions (without jumps) on the unit simplex have already appeared numerous times in the literature. 
In population genetics, prototypical diffusion processes on the unit simplex known as 
\emph{Wright-Fisher diffusions}, or \emph{Kimura diffusions}, 
arise naturally as infinite population limits of discrete Wright-Fisher models 
for allele prevalence in a population of fixed size; see \cite{E:11} for a survey. 
In finance, similar processes have appeared in \cite{GJ:06} under the name of 
\emph{multivariate Jacobi processes}. All these diffusions turn out to be polynomial, and a full characterization is provided in \citet[Section~6.3]{Filipovic/Larsson:2016} by means of necessary and sufficient parameter restrictions on the drift and diffusion coefficients. One could also study other compact state spaces, as has been done in \cite{LP:17}, where polynomial diffusions on compact quadric sets are considered.

As these papers all focus on the case without jumps, it is natural to ask what happens in the jump-diffusion case, 
where the literature is much less developed. This case is considered by \cite{Cuchiero/etal:2012}, however without treating questions of existence, uniqueness, and parameterization for polynomial jump-diffusions on specific state spaces. To analyze these questions on the unit simplex,
the technical difficulties associated with the diffusion case remain, 
arising from the fact that the unit simplex is a non-smooth stratified space \citep[Chapter~1]{EM:13}, 
and that the diffusion coefficient degenerates at the boundary. 
This complicates the analysis, and precludes the use of standard results regarding existence and 
regularity of solutions to the corresponding Kolmogorov backward equations. 
Additionally, in the jump case, the drift and diffusion interact with the (small) jumps orthogonal to the boundary, 
which leads to further mathematical challenges.

Allowing for jumps is however not only of theoretical interest, 
but has practical relevance as well. A concrete illustration of this fact comes from stochastic 
portfolio theory \citep{F:02, FK:09}, where one is interested in the market weights 
$X_i=S_i/(S_1+\cdots+S_d)$ computed from the market capitalizations 
$S_i$, $i=1, \ldots, d$, of the constituents of a large stock index 
such as the S\&P~500 or the MSCI World Index. 
The time evolution of the vector of market weights is thus 
a stochastic process on the unit simplex (see Figure~\ref{fig:1}). 
To model the market weight process, polynomial diffusion models without jumps have 
been found capable of matching certain empirically 
observed features such as typical shape and fluctuations of capital distribution 
curves \citep{FK:05,C:16,CGGPST:16} when calibrated to \emph{jump-cleaned} data. However, the absence of jumps is a deficiency of these models. Indeed, an inspection of market data shows that jumps do occur and 
are an important feature of the dynamics of the market weights.
This is clearly visible in Figure \ref{fig:2} where, for illustrative purposes, 
we have extracted three companies from the MSCI World Index, 
whose market weights exhibit jumps in the period from August 2006 to October 2007.
This application from stochastic portfolio theory underlines the 
importance of specifying jump structures within the polynomial framework. We elaborate on this in Section~\ref{S:SPT_appl}.

\begin{figure}
\begin{minipage}[h]{7cm}
	\centering
	\includegraphics[width=7cm]{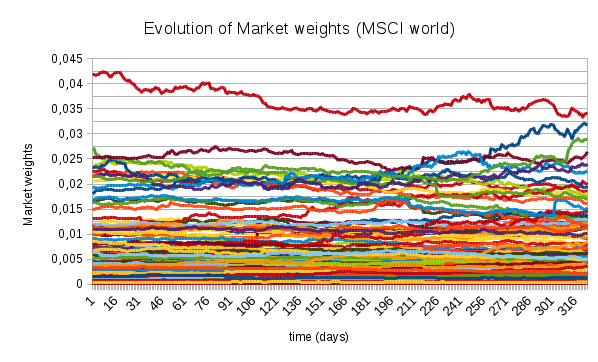}
	\caption{\small{Market weights of the MSCI World Index, August 2006 -- October 2007}}
	\label{fig:1}
\end{minipage}
\hfill
\begin{minipage}[h]{7cm}
	\centering
	\includegraphics[width=7cm]{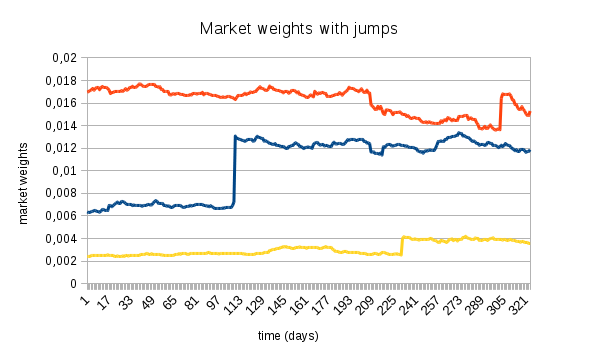}
	\caption{\small{Example of market weights which exhibit jumps}}
	\label{fig:2}
\end{minipage}
\end{figure}

Another natural application of the results developed in this paper arises in default risk modeling following the framework of \cite{Jarrow1998} and \cite{KT:17}. One is then interested in modeling a $[0,1]$-valued stochastic recovery rate which remains at level $1$ for extended periods of time, while occasionally performing excursions away from $1$. Polynomial jump-diffusion specifications turn out to be capable of producing such behavior, while at the same time maintaining tractability.  Further details are given in Section~\ref{S:Def_appl}.

Let us now briefly summarize our main results. Our starting point is a linear operator $\Gcal$ 
whose domain consists of polynomials on a compact state space $E$ 
(initially general, but soon taken to be the unit simplex) and which maps polynomials to polynomials of the same or lower degree. We 
study the corresponding martingale problem for which well-posedness holds if and only 
if $\Gcal$ satisfies the positive maximum principle and is conservative. 
In this case it is of L\'evy type, specified by a diffusion, drift, and jump triplet $(a,b,\nu)$;  
see Theorem~\ref{lem2} and Theorem~\ref{thm3}. We emphasize that not only existence, 
but also uniqueness of solutions to the martingale problem is obtained. 
This is yet another attractive feature of polynomial jump-diffusions on compact state spaces; 
in general, uniqueness is notoriously difficult to establish in the absence of ellipticity or Lipschitz properties. Next, our main focus is on jump specifications with {\em affine jump sizes}, namely,
\begin{equation}\label{C:eq:1} 
\nu(x,A) = \lambda(x) \int \mathds 1_{A\setminus\{0\}}(\gamma(x,y)) \mu(\de y)
\end{equation}
where $\lambda:E\to\R_+$ is a non-negative measurable function, $\mu$ a L\'evy measure and $\gamma=(\gamma_1,\ldots,\gamma_d)$ is of the affine form
\begin{equation}\label{C:eq:2} 
\gamma_i(x,y) = y^0_i + y^1_i x_1 + \cdots + y^d_i x_d;
\end{equation}
see Definition~\ref{D:AJS}. This is the most general specification in the class of jump kernels with 
polynomial dependence on the current state; see Theorem~\ref{T:pol is affine}. 
Under the structural hypothesis of affine jump sizes, we classify all polynomial jump-diffusions on the 
unit interval (i.e.~the unit simplex in $\R^2$); see Theorem~\ref{thm1}. This classification is subsequently 
extended -- under an additional assumption -- to higher dimensions; 
see Theorem~\ref{th:mainsimplex}. Referring to the unit interval for notational convenience, 
we can distinguish four types of jump-diffusions, in addition to the pure diffusion case without jumps: 
\begin{enumerate}
\item[] {\bf Type~1:}  $\lambda$ is constant and the support of $\nu(x,\fdot)$ is contained in $[-x, 1-x]$;
\item[] {\bf Type~2:}  $ \lambda$ is (essentially) a linear-rational function with a pole of order one at the boundary, and the process can only jump in the direction of the pole;
\item[] {\bf Type~3:}  $\lambda$ is (essentially) a quadratic-rational function with a pole of order two in the interior of the state space. There is no jump activity at the pole, but an additional contribution to the diffusion coefficient.
\item[] {\bf Type~4:} $\lambda$ is a quadratic-rational function whose denominator has only complex zeros, and $\mu$ in~\eqref{C:eq:1} is of infinite variation.
\end{enumerate}
This classification already gives an indication of the diversity of possible behavior, 
an impression which is strengthened in Section~\ref{S:examples}, 
where we provide a number of examples both with and without affine jump sizes. 
On the one hand, these examples clearly show that without any structural assumptions like 
\eqref{C:eq:1}--\eqref{C:eq:2}, a full characterization of \emph{all} polynomial jump-diffusions on the simplex, 
or even the unit interval, is out of reach. On the other hand, the examples illustrate the richness 
and flexibility of the polynomial class.

The remainder of the paper is organised as follows. 
Section \ref{sec1.1} summarizes some notation used throughout the article. Section \ref{sec2} is 
concerned with polynomial operators on general compact state spaces and their associated martingale problems. Section \ref{sec3} introduces affine jump sizes. In Section \ref{sec4} we classify all polynomial jump-diffusions on the unit interval with affine jump sizes. It is followed by Section~\ref{S:examples} which deals with examples. Section \ref{sec6} treats the simplex in arbitrary dimension. Finally, Section~\ref{S:appl} discusses applications in stochastic portfolio theory and default risk modeling. Most proofs are gathered in appendices.

\subsection{Notation} \label{sec1.1}

We denote by $\N$ the natural numbers, $\N_0:=\N\cup\{0\}$ the nonnegative integers, and $\R_+$ the nonnegative reals.
The symbols $\R^{d\times d}$, $\S^d$, and $\S^d_+$ denote the $d \times d$ real, real symmetric, 
and real symmetric positive semi-definite matrices, respectively. 
For any subset $E\subseteq\R^d$, we let as usual $C(E)$ denote the space of continuous functions on $E$.  For any sufficiently differentiable function $f$ we write $\nabla f$ for the gradient of $f$ and 
$\nabla^2 f$ for the Hessian of $f$. Next, $e_i$ stands for the $i$th canonical unit vector, $|v|$ denotes the Euclidean norm of the vector $v \in \mathbb{R}^d$, $\delta_{ij}$ is the Kronecker delta, $\delta_x$ is the Dirac mass at $x$, and $\mathbf1$ is the vector whose entries are all equal to 1. 
We denote by $\Pol(\R^d)$ the vector space of all polynomials on $\R^d$ and $\Pol_n(\R^d)$ 
the subspace consisting of polynomials of degree at most~$n$. A \emph{polynomial on $E$} 
is the restriction $p=q|_E$ to $E$ of a polynomial $q\in\Pol(\R^d)$. Its degree is given by
$$\deg p = \min\{\deg q : p = q|_E, q \in \Pol(\R^d)\}.$$
We then let $\Pol(E)$ denote the vector space of polynomials on $E$, and write $\Pol_n(E)$ for those elements whose degree is at most $n$. We frequently use multi-index notation so that, for instance, $x^{\bm k}=x_1^{k_1}\cdots x_d^{k_d}$ for $\bm k=(k_1,\ldots,k_d)\in\N^d_0$.

\section{Polynomial  operators on compact state spaces}\label{sec2}

Let $E\subset\R^d$ be a compact subset of $\R^d$ that will play the role of the state space for a Markov process. 
Later we will specialize to the case where $E$ is the unit interval or the unit simplex. In this paper we are concerned with operators of the following type, along with solutions to the corresponding martingale problems.

\begin{definition}
A linear operator $\Gcal:\Pol(E)\to C(E)$ is called \emph{polynomial} if
\[
\Gcal\big(\Pol_n(E)\big)\subseteq\Pol_n(E) \quad \text{for all $n\in\N_0$.}
\]
\end{definition}

Given a linear operator $\Gcal:\Pol(E)\to C(E)$ and a probability distribution $\rho$ on $E$, 
a \emph{solution to the martingale problem} for $(\Gcal,\rho)$ is a c\`adl\`ag process $X$ with values in $E$ defined on some probability space $(\Omega,\Fcal,\P)$ such that $\P(X_0\in\fdot)=\rho$ and the process $(N_t^f)_{t\geq0}$ given by
\begin{equation}\label{eq:Nf}
N_t^f:=f(X_t)-\int_0^t\Gcal f(X_s)\de s
\end{equation}
is a martingale with respect to the filtration $\Fcal^X_t=\sigma(X_s:s\le t)$ for every $f\in\Pol(E)$. 
We say that {\em the martingale problem for $\Gcal$ is well-posed} if there exists a unique 
(in the sense of probability law) solution to the martingale problem for $(\Gcal,\rho)$ for any initial distribution $\rho$ on $E$. 
If $\Gcal$ is polynomial, then $X$ is called a \emph{polynomial jump-diffusion;} this terminology is justified by Theorem~\ref{thm3} and the subsequent discussion.

\subsection{The positive maximum principle}

\begin{definition}
A linear operator $\Gcal:\Pol(E)\to C(E)$ satisfies the \emph{positive maximum principle} if $\Gcal f(x_0)\leq0$ holds for any $f\in\Pol(E)$ and $x_0\in E$ with $\sup_{x\in E}f(x)=f(x_0)\geq0 $.
\end{definition}

Roughly speaking, the positive maximum principle is equivalent to the existence of solutions to the martingale problem. A typical result in this direction is Theorem~4.5.4 in \cite{EK:05}. For polynomial operators on compact state spaces more is true: we also get uniqueness.

\begin{theorem}\label{lem2}
Let $\Gcal:\Pol(E)\to C(E)$ be a polynomial operator. 
The martingale problem for $\Gcal$ is well-posed if and only if $\Gcal 1=0$ and $\Gcal$ satisfies the positive maximum principle.
\end{theorem}

\begin{proof}
The existence of a solution to the martingale problem for $(\Gcal,\rho)$ 
for any initial distribution $\rho$ on $E$, is guaranteed by Theorem~4.5.4 and Remark~4.5.5 in \cite{EK:05}. 
To prove uniqueness in law, by compactness of $E$ it is enough to prove that the marginal mixed moments of any 
solution $X$ to the martingale problem for $(\Gcal,\rho)$ are uniquely determined by $\Gcal$ and $\rho$; see Lemma 4.1 and Theorem 4.2 in \cite{Filipovic/Larsson:2016}.
To this end, fix any $n\in\N$, let $h_1,\ldots,h_N$ be a basis of $\Pol_n(E)$, and set $H=(h_1,\ldots,h_N)^\top$. 
The operator $\Gcal$ admits a unique matrix representation $G\in\R^{N\times N}$ with respect to this basis, so that
\[
\Gcal p(x)=H(x)^\top G \vec{p},
\]
where $p\in\Pol_n(E)$ has coordinate representation $\vec{p}\in\R^N$, that is, $p(x)=H(x)\vec{p}$; cf.~Section~3 in \cite{Filipovic/Larsson:2016} 
and the proof of Theorem 2.7 in \cite{Cuchiero/etal:2012}. Following the proof of Theorem 3.1 in \cite{Filipovic/Larsson:2016} 
we use the definition of a solution to the martingale problem, linearity of expectation and integration, and the fact that polynomials on the compact set $E$ are bounded, to obtain
\begin{align*}
\vec{p}^{\top} \E[ H(X_T)| \Fcal^X_t]& = \E[p(X_T)| \Fcal^X_t]= p(X_t)+\E\left[\int_t^T \Gcal p(X_s) ds| \Fcal^X_t\right]\\
&=\vec{p}^{\top} H(X_t) +\vec{p}^{\top}G^{\top}\int_t^T \E[H(X_s)| \Fcal^X_t] ds
\end{align*}
for any $t \leq T$ and any $\vec{p}\in\R^N$. For each fixed $t$ this yields a linear integral equation for $\E[ H(X_T)| \Fcal^X_t]$, whose unique solution is $\E[ H(X_T)| \Fcal^X_t]=e^{(T-t)G^{\top}}H(X_t)$. Consequently,
\begin{align}\label{eq:momentformula}
\E[p(X_T ) | \Fcal^X_t]=\vec{p}^{\top}\E[ H(X_T)| \Fcal^X_t]=H(X_t)^\top e^{(T-t)G}\vec{p},
\end{align}
which in particular shows that all marginal mixed moments are uniquely determined by $\Gcal$ and $\rho$, as required.

For the converse implication, observe that since every solution to the martingale problem 
for $(\Gcal,\rho)$ is conservative, the condition $\Gcal1=0$ follows directly by the martingale property 
of \eqref{eq:Nf} with $f=1$. The necessity of the positive maximum principle is standard; see for instance the proof of Lemma 2.3 in \cite{Filipovic/Larsson:2016}.
\end{proof}

\begin{remark}
 Observe that a solution to the martingale problem is conservative by definition  since it is supposed to take values in $E$. 
 This is reflected by the condition $\Gcal1=0$ of Theorem~\ref{lem2} and in the definition of L\'evy type 
 operator in the next section. Let us remark that the condition in Theorem~\ref{lem2}, namely that the positive maximum principle 
and $\Gcal 1=0$ are satisfied, is equivalent to the \emph{maximum principle}, that is, 
 $\Gcal f(x_0)\leq0$ holds for any $f\in\Pol(E)$ and any $x_0\in E$ with $\sup_{x\in E}f(x)=f(x_0)$.
\end{remark}

\begin{remark}
While existence of a solution to the martingale problem is equivalent to the maximum principle in very general settings,
it is remarkable that in the case of polynomial operators on compact state spaces uniqueness also follows. Indeed, without the assumption that $\Gcal$ is polynomial, it is well-known that the maximum principle is not enough to guarantee uniqueness. 
For example, with $E=[0,1]$ and $\Gcal f(x) = \sqrt{x}(1-x)f'(x)$, 
the functions $X_t = (e^t-1)^2/(e^t+1)^2$ and $X_t\equiv 0$ are two different solutions to the martingale problem 
for $(\Gcal,\delta_{0})$. In the polynomial case, well-posedness is deduced from uniqueness of moments, which is a consequence of~\eqref{eq:momentformula}. Let us emphasize that \eqref{eq:momentformula} gives more than mere uniqueness: it gives an explicit formula for computing the moments via a matrix exponential. This tractability is crucial in applications, and was used as a defining property of this class of processes in \cite{Cuchiero/etal:2012}.
\end{remark}

\subsection{L\'evy type representation}

 \begin{definition}
An operator $\Gcal:\Pol(E)\to C(E)$ is said to be of \emph{L\'evy type} if it can be represented as
\begin{equation} \label{eq:G LK}
\begin{aligned}
\Gcal f(x) &= \frac{1}{2}\tr\left( a(x) \nabla^2 f(x) \right) + b(x)^\top \nabla f(x) \\
&\qquad + \int \left( f(x + \xi) - f(x) - \xi^\top \nabla f(x) \right) \nu(x,\de\xi),
\end{aligned}
\end{equation}
where the right-hand side can be computed using an arbitrary representative of $f$, and the triplet $(a,b,\nu)$ consists of bounded measurable functions $a:E\to\S^d_+$ and $b:E\to\R^d$, and a kernel $\nu(x,\de\xi)$ from $E$ into $\R^d$ satisfying
\begin{equation}\label{thenu}
\sup_{x\in E} \int |\xi|^2 \nu(x,\de\xi) < \infty, \quad \nu(x,\{0\})=0, \quad \nu(x, (E-x)^c )=0 \quad\text{for all $x\in E$.}
\end{equation}
\end{definition}

Polynomial operators satisfying the positive maximum principle are always L\'evy type operators, 
as is shown in Theorem~\ref{thm3} below. This parallels known results regarding operators acting on smooth and compactly supported functions, 
see \cite{C:65} or \citet[Theorem 2.21]{BSW:13} for Feller generators, and also \cite{Hoh:1998}. A crucial ingredient in the proof of Theorem~\ref{thm3} is the classical Riesz-Haviland theorem, which we now state. A proof can be found in \cite{Haviland:1935} and (\citeyear{Havliand:1936}), or e.g. \cite{Marshall:2008}.

\begin{lemma}[Riesz-Haviland]\label{lem5}
Let $K\subset\R^d$ be compact, and consider a linear functional $\Wcal:\Pol(K)\to\R$. Then the following conditions are equivalent.
\begin{itemize}
\item[(i)] $\Wcal(f)=\int f(\xi)\mu(\de\xi)$ for all $f\in\Pol(K)$ and a Borel measure $\mu$ concentrated on $K$.
\item[(ii)] $\Wcal(f)\geq0$ for all $f\in\Pol(K)$ such that $f\geq0$ on $E$.
\end{itemize}
\end{lemma}

We now state Theorem~\ref{thm3} regarding the L\'evy type representation of operators satisfying the positive maximum principle. The proof is given in Section \ref{app11}.

\begin{theorem}\label{thm3}
Consider a linear operator $\Gcal:\Pol(E)\to C(E)$. If $\Gcal 1=0$ and $\Gcal$ satisfies the positive maximum principle, then $\Gcal$ is a L\'evy type operator. \end{theorem}

Suppose $\Gcal:\Pol(E)\to C(E)$ is a linear operator with $\Gcal 1=0$ that satisfies the positive maximum principle, and let $X$ be a solution to the associated martingale problem. Then $X$ is a semimartingale, as can be seen by taking $f(x)=x_i$ in~\eqref{eq:Nf}. We claim that its diffusion, drift, and jump characteristics (with the identity map as truncation function) are given by
\[
\int_0^t a(X_s)\de s, \qquad \int_0^t b(X_s)\de s, \qquad \nu(X_{t-},\de\xi)\de t,
\]
where $(a,b,\nu)$ is the triplet of the L\'evy type representation~\eqref{eq:G LK}. 
To see this, first note that $\Gcal$ can be extended to $C^2$ functions on $E$ using~\eqref{eq:G LK}. 
Then, an approximation argument shows that $N^f$ in~\eqref{eq:Nf} remains a martingale 
for such functions~$f$. The claimed form of the 
characteristics of $X$ now follows from Theorem~II.2.42 in~\cite{Jacod/Shiryaev:2003}; 
see also Proposition 2.12 in \cite{Cuchiero/etal:2012}. This justifies referring to $X$ as a polynomial jump-diffusion. 
Since  the martingale problem is well-posed by Theorem~\ref{lem2}, such a 
polynomial jump-diffusion is a Markov process, and hence a polynomial process in the 
sense of \cite{Cuchiero/etal:2012}.

The following lemma provides necessary and sufficient conditions on the triplet $(a,b,\nu)$ in order that $\Gcal$ be polynomial.

\begin{lemma}\label{lem1}
Let $\Gcal:\Pol(E)\to C(E)$ be a L\'evy type operator with triplet $(a,b,\nu)$. Then $\Gcal$ is polynomial  if and only if
\[
b_i\in\Pol_1(E), \qquad
a_{ij}+\int\xi_i\xi_j\nu(\fdot,\de\xi)\in\Pol_2(E), \qquad
\int\xi^{\bk}\nu(\fdot,\de\xi)\in\Pol_{|\bk|}(E)
\]
for all $i,j\in\{1,\ldots,d\}$ and $|\bk|\geq3$.
\end{lemma}

\begin{proof}
This result is well-known, see for instance \cite{Cuchiero/etal:2012}, and the proof is simple. Indeed, direct computation yields $0=\Gcal (1)(x)$, $b_i(x)=\Gcal( e_i^\top (\fdot - x))(x)$, $a_{ij}(x)+\int \xi_i\xi_j \nu(x,\de\xi)=\Gcal( e_i^\top (\fdot - x) e_j^\top (\fdot - x))(x)$, and $\int \xi^{\bm k} \nu(x,\de\xi)=\Gcal( (\fdot - x)^{\bm k} )(x)$ for $|\bm k|\ge 3$. Thus, if $\Gcal$ is polynomial, one can show that the triplet satisfies the stated conditions. The converse implication is immediate from the observation that $\deg(pq)\le\deg(p)+\deg(q)$ for any $p,q\in\Pol(E)$.
\end{proof}

\subsection{Conic combinations of polynomial operators}

Due to Theorem~\ref{lem2} and Theorem~\ref{thm3}, every member of the set
\[
\Kcal := \{ \Gcal: \Pol(E)\to C(E) \colon \text{$\Gcal$ is polynomial and its martingale problem is well-posed}\}
\]
is of L\'evy type~\eqref{eq:G LK}. The set $\Kcal$ also possesses the following stability properties, 
which are useful for constructing examples of polynomial jump-diffusions; 
we do this in Section~\ref{S:examples}. The proofs of the following two results are given in Section \ref{appa2}.

\begin{theorem}\label{lem3}
The set $\Kcal$ is a convex cone closed under pointwise convergence, in the sense that if $\Gcal_n\in\Kcal$ for $n\in\N$ and $\Gcal f(x):=\lim_{n\to\infty}\Gcal_n f(x)$ exists and is finite for all $f\in\Pol(E)$ and $x\in E$, then $\Gcal\in \Kcal$.
\end{theorem}

If an operator $\Gcal$ is the limit of $\Gcal_n$ as in Theorem~\ref{lem3}, then its triplet $(a,b,\nu)$ can be expressed in terms of the triplets $(a^n,b^n,\nu^n)$ of the operators $\Gcal_n$.

\begin{lemma}\label{lem7}
Suppose that $\Gcal_n\in\Kcal$, and let $a^n$, $b^n$, and $\nu^n(x,\de\xi)$ be the coefficients of its L\'evy type representation, for all $n\in \N$.  Then $\Gcal f (x):=\lim_{n\to\infty}\Gcal_n f (x)$ exists and is finite for all $f\in\Pol(E)$ and $x\in E$ if and only if the coefficients
\[
b_i^n,\qquad a_{ij}^n+\int \xi_i\xi_j\nu^n(\fdot,\de\xi),\qquad \int \xi^\bk\nu^n(\fdot,\de\xi)
\] 
converge pointwise as $n\to\infty$ for all  $i,j\in\{1,\ldots,d\}$ and $|\bk|\geq3$. In this case the triplet $(a,b,\nu)$ of the L\'evy type representation of $\Gcal$ is given by
\[
b_i(x)=\lim_{n\to\infty}b_i^n(x),\quad a_{ij}(x)=\lim_{n\to\infty}\bigg(a_{ij}^n(x)+\int\xi_i\xi_j\nu^n(x,\de\xi)\bigg)-\int\xi_i\xi_j\nu(x,\de\xi),
\] 
for all $x\in E$ and $i,j\in\{1,\ldots,d\}$, where the kernel $\nu(x,\de\xi)$ is uniquely determined by
$$\int \xi^\bk\nu(x,\de\xi)=\lim_{n\to\infty}\int \xi^\bk\nu^n(x,\de\xi), \qquad |\bk|\geq3.$$
\end{lemma}

\begin{remark}
The diffusion coefficient $a(x)$ is the limit of $a^n(x)$ if and only if the weak limit of $|\xi|^2\nu^n(x,\de\xi)$ exists and has no mass in zero. If the weak limit does have mass in zero, then this mass is equal to the difference between $\Tr(a(x))$ and the limit of $\Tr(a^n(x))$.
\end{remark}

\section{Affine and polynomial jump sizes}\label{sec3}

Throughout this section we continue to consider a compact state space $E\subset \R^d$.
In the absence of jumps it is relatively straightforward to explicitly write down a complete parametrization of polynomial diffusions on the unit interval or the unit simplex; see \cite{Filipovic/Larsson:2016}. With jumps this is no longer the case. Indeed, examples in Section~\ref{S:examples} 
illustrate the diversity of behavior that is possible even on the simplest nontrivial state space $[0,1]$. Therefore, in order to make progress we will restrict attention to specifications whose jumps are of the following state-dependent type.
Consider a jump kernel $\nu(x,\de\xi)$ from $E$ into $\R^d$ satisfying (\ref{thenu}).

\begin{definition} \label{D:AJS}
The jump kernel $\nu(x,\de\xi)$ is said to have {\em affine jump sizes} if it is of the form
\begin{equation} \label{eq:nu(x,A)}
\nu(x,A) = \lambda(x) \int \mathds1_{A\setminus\{0\}}(\gamma(x,y)) \mu(\de y)
\end{equation}
where $\lambda:E\to\R_+$ is a nonnegative measurable function, $\gamma=(\gamma_1,\ldots,\gamma_d)$ is of the affine form
\begin{equation} \label{eq:gamma affine}
\gamma_i(x,y) = y^0_i + y^1_i x_1 + \cdots + y^d_i x_d,
\end{equation}
and $\mu(\de y)$ is a measure on $\R^{d(d+1)}$ satisfying $\int (|y|^2\wedge 1)\mu(\de y)<\infty$. Here we use the notation $y=(y^j_i: i=1,\ldots,d,\, j=0,\ldots,d)\in\R^{d(d+1)}$ for the vector of coefficients appearing in~\eqref{eq:gamma affine}.
\end{definition}

\begin{remark}\label{rem3}
By  (\ref{thenu}) and compactness of $E$, 
the measure $\mu(\de y)$ can always be chosen compactly supported. 
In this case, all its moments of order at least two are finite.
\end{remark}

Intuitively, \eqref{eq:nu(x,A)} means that the conditional distribution of the jump $\Delta X_t$, given that it is nonzero and the location immediately before the jump is $X_{t-}=x$, is the same as the distribution of $\gamma(x,y)$ under $\mu(\de y)$; at least when $\mu(\de y)$ is a probability measure. The jump intensity is state-dependent and given by $\nu(x,\R^d)=\lambda(x)\mu(\{\gamma(x,\fdot)\ne 0\})$, which may or may not be finite. 

Jump kernels with affine jump sizes can be used as building blocks to obtain a large class of specifications by means of Theorem~\ref{lem3}. The jump kernels obtained in this way are of the form
\[
\nu(x,\de\xi) = \sum_k \nu_k(x,\de\xi),
\]
where each jump kernel $\nu_k(x,\de\xi)$ has affine jump sizes. We refer to such specifications as having {\em mixed affine jump sizes}.

The affine form of $\gamma(x,y)$ is a particular case of the seemingly more general situation where $\gamma(x,y)$ 
is allowed to depend polynomially on the current state~$x$. However, 
this would not actually lead to an increase in generality in the context of polynomial jump-diffusions. 
Indeed, at least in the case when $E$ has nonempty relative interior in its affine hull, the following result shows that whenever jump sizes are polynomial, they are necessarily affine. The proof is given in Section \ref{secc}.

\begin{theorem} \label{T:pol is affine}
Assume that $E$ has nonempty relative interior in its affine hull. Let $\nu(x,\de\xi)$ be a jump kernel from $E$ into $\R^d$ of the form~\eqref{eq:nu(x,A)} and satisfying \eqref{thenu}, where $\lambda$ is nonnegative and measurable, $\gamma$ is given by
\[
\gamma_i(x,y) = \sum_{|\bk|\le K} y^i_{\bk} x^{\bk}
\]
for some $K\in\N_0$, and $\mu(\de y)$ is a measure on $(\R^d)^{\dim\Pol_K(\R^d)}$ with $\int (|y|^2\wedge 1)\mu(\de y)<\infty$. Assume also that $\nu(x,\de\xi)$ satisfies
\begin{equation} \label{eq:pol is affine:2}
\int \xi^{\bm k} \nu(\fdot,\de\xi) \in \Pol_{|\bm k|}(E), \qquad |\bm k|\ge3,
\end{equation}
and that $E$ has nonempty interior. Then one can choose $\mu(\de y)$ so that $y^i_{\bm k}=0$ a.e.~for all $i=1,\ldots,d$ and all $|\bm k|\ge 2$. That is, $\nu(x,\de\xi)$ has affine jump sizes.
\end{theorem}

\begin{remark}\label{rem10}
Note that if $\nu(x,\de\xi)$ has affine jump sizes and satisfies~\eqref{eq:pol is affine:2}, then the function $\lambda$ is can be expressed as the ratio of two polynomials of degree at most four,
\[
\lambda(x) = \frac{\int |\xi|^4 \nu(x,\de\xi)}{\int |\gamma(x,y)|^4 \mu(\de y)},
\]
at points $x$ where the denominator is nonzero. At points $x$ where the denominator vanishes, we have $\gamma(x,y)=0$ for $\mu$-a.e.~$y$, whence $\nu(x,\de\xi)=0$ due to \eqref{eq:nu(x,A)}. Thus we may always take $\lambda(x)=0$ at such points.
\end{remark}

\begin{remark}
Jump specifications of the form~\eqref{eq:nu(x,A)} are convenient from the point of view of representing 
solutions $X$ to the martingale problem 
for $\Gcal$ as solutions to stochastic differential equations driven by a Brownian motion and a 
Poisson random measure. Indeed,
such a stochastic differential equation has the following form:
\begin{align*}
 X_t= X_0&+ \int_0^t b(X_s) ds + \int_0^t \sqrt{a(X_s)} dW_s\\
&\qquad + \int_0^t \int_0^{\lambda(X_{s-})} \int \gamma(X_{s-},y) \left( N(ds, du, dy)- ds du\mu(dy) \right),
\end{align*}
where $\sqrt{\fdot}$ denotes the matrix square root, $W$ is a $d$-dimensional Brownian motion and $N(ds, du, dy)$ is a Poisson random measure
on $\mathbb{R}_+^2 \times \supp (\mu)$ whose intensity measure is $ds du\mu(dy)$. See also, for instance, \citet[Section 5]{DL:06}, regarding analogous representations of affine processes. Note that a representation of the form~\eqref{eq:nu(x,A)} always exists, even with $\lambda\equiv1$, if one allows $y$ to lie a suitable Blackwell space; see \citet[Remark~III.2.28]{Jacod/Shiryaev:2003}. Thus, in view of Theorem~\ref{T:pol is affine}, our restriction to affine jump sizes in the sense of Definition~\ref{D:AJS} is essentially equivalent to a polynomial dependence of $\gamma(x,y)$ on $x$, somewhat generalized by allowing a state dependent intensity $\lambda(x)$. Note also that once $\gamma(x,y)$ depends polynomially on $x$, there is no loss of generality to assume that $y$ lies in an Euclidean space.

\end{remark}


\section{The unit interval}\label{sec4}

Throughout this section we consider the state space
\[
E := [0,1].
\]
Our goal is to characterize all polynomial jump-diffusions on~$E$ with affine jump sizes. The general existence and uniqueness result Theorem~\ref{lem2}, in conjunction with Lemma~\ref{lem1}, leads to the following refinement of Theorem \ref{thm3}, characterizing those triplets $(a,b,\nu)$ that correspond to polynomial jump-diffusions. The proof is given in Section~\ref{app21}.

\begin{lemma}\label{thm4}
A linear operator $\Gcal:\Pol(E)\to C(E)$ is polynomial  and its martingale problem is well-posed if and only if it is of form \eqref{eq:G LK}
and the corresponding triplet $(a,b,\nu)$ satisfies
\begin{enumerate}
\item\label{thm4:1} $a\geq0$ and $\nu(x,\de\xi)$ satisfies \eqref{thenu},
\item\label{thm4:2} $a(0)=a(1)=0$, $b(0)-\int\xi\nu(0,\de\xi)\geq0$, and $b(1)-\int\xi\nu(1,\de\xi)\leq0$,
\item\label{thm4:3} $b\in\Pol_1(E)$,
 $a+\int\xi^2\nu(\fdot,\de\xi)\in\Pol_2(E)$, and 
 $\int\xi^n\nu(\fdot,\de\xi)\in\Pol_n(E)$ for all $n\geq3$.
\end{enumerate}
\end{lemma}
Observe that condition (i) guarantees that $\Gcal$ is of L\'evy Type.
\begin{remark}
Condition~\ref{thm4:2} implies that $\int|\xi|\nu(x,\de\xi)<\infty$ for $x\in\{0,1\}$. Intuitively, this means that the solution to the martingale problem for $\Gcal$ has a purely discontinuous martingale part which is necessarily of finite variation on the boundary of~$E$.
\end{remark}

We now turn to the setting of affine jump sizes in the sense of Definition~\ref{D:AJS}. We thus consider L\'evy type operators $\Gcal$ of the form
\begin{equation} \label{eqn52}
\begin{aligned}
&\Gcal f(x) = \frac 1 2 a(x)f''(x)+ b(x) f'(x) \\
& \qquad\qquad +\lambda(x) \int \left( f(x+\gamma(x,y))-f(x)- \gamma(x,y) f'(x) \right) \mu(\de y), \\
\end{aligned}
\end{equation}
where $\lambda$ is nonnegative and measurable, and $\gamma(x,y)$ is affine in $x$. The main result of this section, Theorem~\ref{thm1} below, shows that the generator must be of one of five mutually exclusive types, which we now describe.

\paragraph{Type~0.} Let $a(x)=Ax(1-x)$, $b(x)=\kappa(\theta-x)$, where $A\in\R_+$, $\kappa\in\R_+$, 
and $\theta\in[0,1]$, and set $\lambda=0$. Then $\Gcal$ is a polynomial operator whose martingale problem 
is well-posed. The solution $X$ corresponds simply to the well-known Jacobi diffusion, 
which is the most general polynomial diffusion on the unit interval.

\paragraph{Type~1.} Let $a(x)=Ax(1-x)$, $b(x)=\kappa(\theta-x)$, and $\lambda(x)=1$, where $A\in\R_+$, $\kappa\in\R_+$, and $\theta\in[0,1]$. Furthermore, writing $y=(y_1,y_2)$ we define $\gamma(x,y)=y_1(-x)+y_2(1-x)$ and let $\mu$ be a nonzero measure on $[0,1]^2\setminus\{0\}$. If the boundary conditions
\[
\kappa\theta \ge  \int y_2 \mu(\de y) \qquad\text{and}\qquad \kappa(1-\theta) \ge   \int y_1 \mu(\de y)
\]
are satisfied, then $\Gcal$ is a polynomial operator whose martingale problem is well-posed.

Note that the boundary conditions imply that $\int |\xi| \nu(x,\de\xi) \le 2\int |y| \mu(\de y)$ is bounded. 
Thus, the resulting process behaves like a Jacobi diffusion with summable jumps. The arrival intensity of the jumps is $\nu(x,E-x)=\mu(\{y: \gamma(x,y)\ne0\})$, which may or may not be finite. Figure~\ref{figc1} illustrates the form of $a$, $\lambda$ 
and the support $\gamma(x,y)$ under $\mu$.

\paragraph{Type~2.} Let $a(x)=Ax(1-x)$, $b(x)=\kappa(\theta-x)$, and $\lambda(x)=\frac{1}{x}(1+q x)\mathds 1_{\{x\ne 0\}}$ where $A\in\R_+$, $\kappa\in\R_+$, $\theta\in[0,1]$, and $q \in [-1,\infty)$. Furthermore, define $\gamma(x,y)=-xy$ and let $\mu$ be a nonzero square-integrable measure on $(0,1]$. Notice that $y$ is scalar. If the boundary condition
\[
\kappa(1-\theta) \ge (1+q) \int y \mu(\de y)
\]
is satisfied, then $\Gcal$ is a polynomial operator whose martingale problem is well-posed.

The boundary condition implies, if $q>-1$, that $\int |\xi| \nu(x,\de\xi) \le (1+|q|) \int y\mu(\de y)$ is bounded.
Thus, in this case, the solution $X$ to the martingale problem for $\Gcal$ has summable jumps. If $q=-1$, the jumps need not be summable. The arrival intensity of the jumps is $\nu(x,E-x)=\lambda(x)\mu((0,1])$ and hence, even if $\mu$ is a finite measure, the jump intensity is unbounded around $x=0$. Moreover, due to the form of $\gamma(x,y)$, $X$ can only jump to the left, and since $\nu(0,E)=0$, $X$ cannot leave $x=0$ by means of a jump. Figure~\ref{figc2} illustrates the form of $a$, $\lambda$ and the support $\gamma(x,y)$ under $\mu$.

By reflecting the state space around the point $1/2$, we obtain a similar structure which we also classify as Type~2, where now the jump intensity is unbounded around $x=1$. The diffusion and drift coefficients remain as before, while $\lambda(x)=\frac{1}{1-x}(1+q(1-x))\mathds 1_{\{x\ne 1\}}$ for some $q\in[-1,\infty)$, the jump sizes are $\gamma(x,y)=(1-x)y$, and $\mu$ is a nonzero square-integrable measure on $(0,1]$ as before. The boundary condition becomes $\kappa\theta \ge (1+q) \int y \mu(\de y)$.

\paragraph{Type~3.} Let $x^*\in(0,1)$; this will be a ``no-jump'' point. Let $b(x)=\kappa(\theta-x)$ and set
\[
\lambda(x)=\frac{q_0+q_1 x + q_2 x^2}{(x-x^*)^2}\mathds 1_{\{x\ne x^*\}},
\]
where $\kappa\in\R_+$, $\theta\in[0,1]$, and $q_0,q_1,q_2$ are real numbers such that the numerator of $\lambda$ is nonnegative on $E$ without zeros at $x^*$. Furthermore, define $\gamma(x,y)=-(x-x^*)y$, and let $\mu$ be a nonzero square-integrable measure on $(0,(x^*\vee(1-x^*))^{-1}]$. Finally, let $a(x)=Ax(1-x) + a^\nu\mathds 1_{\{x=x^*\}}$ where
\[
a^\nu = \left(q_0 + q_1 x^* + q_2 (x^*)^2\right) \int y^2 \mu(\de y).
\]
If the boundary conditions
\[
\kappa\theta \ge \frac{q_0}{x^* } \int y \mu(\de y) \quad\text{and}\quad \kappa(1-\theta) \ge \frac{q_0+q_1+q_2}{1-x^*} \int y \mu(\de y)
\]
are satisfied, then $\Gcal$ is a polynomial operator whose martingale problem is well-posed.

If $q_0 + q_1x + q_2 x^2 = Lx(1-x)$ for some constant $L\in\R_+$, the solution $X$ to the martingale problem for $\Gcal$ may have non-summable jumps. If the numerator of $\lambda(x)$ is not of this form, then the boundary conditions imply that $X$ has summable jumps. The arrival intensity of the jumps is
$$\nu(x,E-x)=\lambda(x)\mu\left(\left(0,\frac 1 {x^*}\wedge \frac 1 {1-x^*}\right]\right).$$
 As a result, even if $\mu$ is a finite measure, the jump intensity has a pole of order two at $x=x^*$, which results in a contribution of size $a^\nu$ to the diffusion coefficient. Moreover, due to the form of $\gamma(x,y)$, the jumps of $X$ are always in the direction of the ``no-jump'' point $x^*$. Although the jumps may overshoot~$x^*$, they always serve to reduce the distance to $x^*$. In particular, since $\nu(x^*,E-x^*)=0$, $X$ cannot leave $x=x^*$ by means of a jump. Figure~\ref{figc3} illustrates the form of $a$, $\lambda$ and the support $\gamma(x,y)$ under $\mu$.

\paragraph{Type~4.}
Suppose $\alpha\in\C\setminus\R$ is a non-real complex number such that $|2\alpha-1|<1$ and let $\mu$ be a nonzero square-integrable measure on $[0,1]\times[0,1]$ such that
\begin{equation}\label{case4}
\int\big(y_1(-\alpha)+y_2(1-\alpha)\big)^n\mu(\de y)=0, \qquad n\ge 2,
\end{equation}
and $\int y_1\mu(\de y)=\int y_2\mu(\de y)=\infty$. Let $b(x)=\kappa(\theta-x)$ and set
\[
\lambda(x)=\frac{Lx(1-x)}{(x-\alpha)(x-\overline\alpha)},
\]
where $\kappa\in\R_+$, $\theta\in[0,1]$, and $L>0$. Furthermore, define $\gamma(x,y)=y_1(-x)+y_2(1-x)$ and let
$
a(x) = Ax(1-x)$
for some $A\in\R_+$. Then $\Gcal$ is a polynomial operator whose martingale problem is well-posed. \\

Having described five types of processes which are polynomial jump-diffusions due to the conditions of Lemma~\ref{thm4}, we are now ready to state the
converse result, namely that \emph{all} polynomial jump-diffusions on $[0,1]$ with affine jump sizes are necessarily of one of these types. The proof is given in Section~\ref{app21}.

\begin{theorem}\label{thm1}
Let $\Gcal$ be a polynomial operator whose martingale problem is well-posed. If the associated jump kernel has affine jump sizes, then $\Gcal$ necessarily belongs to one of the Types~0-4.
\end{theorem}

\begin{remark}\label{rem:nonexistence}
Let us end this section with some remarks regarding Type~4. First, note that  $\int y_1\mu(\de y)=\int y_2\mu(\de y)=\infty$ implies that $\mu(\de y)$ cannot be a product measure 
since in this case $\int y_1 y_2 \mu(\de y)$ would be infinite too, which however contradicts square integrability.
Second, after passing to polar coordinates $(r, \varphi)$, the condition \eqref{case4} becomes
\begin{equation}\label{eqn8}
\int_{[0,\pi]\times\R_+}r^ne^{{\rm i}n\varphi} \mu_\alpha(\de \varphi,\de r)=0,\qquad n\ge2,
\end{equation}
where $\mu_\alpha$ is the compactly supported measure given by
$$\mu_\alpha(A):=\int\mathds1_A\Big(\Arg\big(y_1(-\alpha)+y_2(1-\alpha)\big), \big|y_1(-\alpha)+y_2(1-\alpha)\big|\Big)\mu(\de y)$$
for all measurable subsets $A\subseteq [0,\pi]\times\R_+$. It can then be shown that $r$ and $\varphi$ cannot be independent, i.e., $\mu_\alpha$ cannot be a product measure. These observations indicates that natural attempts to find 
combinations of $\alpha$ and $\mu$ satisfying \eqref{case4} do not work. 
In fact, it is unknown to us what a potential example of Type~4 might look like.
Note also that Type~4 is distinct from all other types in the following respect. For Types~1--3, $\lambda \gamma^n(\fdot, y)$ is a polynomial on $E$ (outside the ``no-jump'' point) 
of degree $n\geq 2$ for all $y \in \supp(\mu)$,
whereas for Type~4 this property holds true only for the integrated quantity $\lambda \int \gamma(\fdot,y)^n \mu(dy)$. 
\end{remark}

\begin{figure}[h!]
\begin{center}
      \begin{minipage}{0.3\linewidth}
\begin{figure}[H]
\begin{tikzpicture}[scale=0.9,
	declare function = {A(\x) =1-\x;},
	declare function = {B(\x) =-\x;},
        declare function = {a(\x) =4*\x*(1-\x);},
        declare function = {l(\x) =1/5;},
	]
    \begin{axis}[
        domain=0:1,
        y domain=-1:1,
        samples=300,
        axis on top=true,
        compat=1.3,
     ytick={-1,1}, 
     xtick={0,1},
      extra x ticks={0,100},
     yticklabels={-1,1},
     xticklabels={0,1},
y tick style={black, thick},
    x tick style={black, thick},
    axis y line* =middle,
    axis x line*=middle,
      x=4cm, y=4cm, z=0cm,
      y axis line style={black!20!white},
        xmin=0,
        xmax=1,
        ymin=-1,
        ymax=1,
        enlarge y limits,
        enlarge x limits,
        ]
        \draw[black!20!white] (100,-20)--(100,220);
       \fill[gray, opacity=0.05222](0,200)--(100,100)--(100,0)--(0,100)--(0,200);
\fill[green!50!gray, opacity=0.2](0,200)--(100,100)--(100,0)--(0,100)--(0,200);
        \addplot[thick,red] {a(x)};
        \addplot[thick,blue] {l(x)};
        \draw (80,100) node[inner sep=0pt] (j1) {};  
       \draw (80-60,100)[inner sep=0pt]  node (j2) {}; 
       \draw (80-40,100)[inner sep=0pt]  node (j3) {}; 
       \draw (80-20,100)[inner sep=0pt]  node (j4) {}; 
       \draw (80+14,100)[inner sep=0pt]  node (j5) {}; 
       \draw[->,black!40!green] (j1) to [bend left=80] (j2);
       \draw[->,black!40!green] (j1) to [bend left=80] (j3);
       \draw[->,black!40!green] (j1) to [bend left=80] (j4);
       \draw[->,black!40!green] (j1) to [bend right=80] (j5);
\draw[black!30!green, very thick] (0,100) -- (100,100) node[anchor=north]{};

 \end{axis}[]
\node[align=center, yshift=-1em] (title) 
    at (current bounding box.south)
    {Figure \ref{figc1}};
\end{tikzpicture}
\end{figure}
 \end{minipage}
      \begin{minipage}{0.3\linewidth}
      \begin{figure}[H]
\begin{tikzpicture}[scale=0.9,
	declare function = {A(\x) =1-\x;},
	declare function = {B(\x) =-\x;},
	declare function = {J(\x) =0;},
        declare function = {a(\x) =4*\x*(1-\x);},
        declare function = {l(\x) =0.01*(\x+4)/(\x+0.01);},
	]
    \begin{axis}[
        domain=0:1,
        y domain=-1:1,
        samples=300,
        axis on top=true,
        compat=1.3,
     ytick={-1,1}, 
     xtick={0,1},
      extra x ticks={0,20},
     yticklabels={-1,1},
     xticklabels={0,1},
y tick style={black, thick},
    x tick style={black, thick},
    axis y line* =middle,
    axis x line*=middle,
      x=4cm, y=4cm, z=0cm,
      y axis line style={black!20!white},
        xmin=0,
        xmax=1,
        ymin=-1,
        ymax=1,
        enlarge y limits,
        enlarge x limits,
        ]
        \draw[black!20!white] (100,-20)--(100,220);
\fill[gray, opacity=0.05222](0,200)--(100,100)--(100,0)--(0,100)--(0,200);
\fill[green!50!gray, opacity=0.2](100,100)--(100,0)--(0,100)--(100,100);
        \addplot[thick,red] {a(x)};
        \addplot[thick,blue] {l(x)};
        \draw (80,100) node[inner sep=0pt] (j1) {};  
       \draw (80-60,100)[inner sep=0pt]  node (j2) {}; 
       \draw (80-40,100)[inner sep=0pt]  node (j3) {}; 
       \draw (80-20,100)[inner sep=0pt]  node (j4) {}; 
       \draw[->,black!40!green] (j1) to [bend left=80] (j2);
       \draw[->,black!40!green] (j1) to [bend left=80] (j3);
       \draw[->,black!40!green] (j1) to [bend left=80] (j4);
\draw[black!30!green,very thick] (0,100) -- (80,100) node[anchor=north]{};
        \filldraw [blue] (0,100) circle (1pt) node[blue] {};

  \end{axis}[]
\node[align=center, yshift=-1em] (title) 
    at (current bounding box.south)
    {Figure \ref{figc2}};
\end{tikzpicture}
\end{figure}
 \end{minipage}
      \begin{minipage}{0.3\linewidth}
      \begin{figure}[H]
\begin{tikzpicture}[scale=0.9,
	declare function = {A(\x) =1-\x;},
	declare function = {B(\x) =-\x;},
	declare function = {J1(\x) =0;},
        declare function = {J2(\x) =1-5*\x/3;},
        declare function = {a(\x) =4*\x*(1-\x);},
        declare function = {l1(\x) =0.0035*(\x+5)/(\x-1+2/5-0.1)/(\x-1+2/5-0.1);},
        declare function = {l2(\x) =0.0035*(\x+5)/(\x-1+2/5+0.1)/(\x-1+2/5+0.1);},
	]
    \begin{axis}[
        domain=0:1,
        y domain=-1:1,
        samples=300,
        axis on top=true,
        compat=1.3,
     ytick={-1,1}, 
     xtick={0,0.6,1},
      extra x ticks={0},
     yticklabels={-1,1},
     xticklabels={0,$x^*$,1},
y tick style={black, thick},
    x tick style={black, thick},
    axis y line* =middle,
    axis x line*=middle,
      x=4cm, y=4cm, z=0cm,
      y axis line style={black!20!white},
        xmin=0,
        xmax=1,
        ymin=-1,
        ymax=1,
        enlarge y limits,
        enlarge x limits,
        ]
        \draw[black!20!white] (100,-20)--(100,220);
\fill[gray, opacity=0.05222](0,200)--(100,100)--(100,0)--(0,100)--(0,200);
\fill[green!50!gray, opacity=0.2](0,200)--(60,100)--(0,100)--(0,200);
\fill[green!50!gray, opacity=0.2](100,0)--(60,100)--(100,100)--(100,0);
        \addplot[thick,red,domain=0:0.58] {a(x)};
               \addplot[thick,red,domain=0.62:1] {a(x)};
        \addplot[thick,blue,domain=0:0.58] {l1(x)};
              \addplot[thick,blue, domain=0.62:1] {l2(x)};
       \draw (40,100) node[inner sep=0pt] (j1) {};  
       \draw (40+200 -500/3,100)[inner sep=0pt]  node (j2) {}; 
       \draw (40+180-500/3,100)[inner sep=0pt]  node (j3) {}; 
       \draw (40+94 -500/3+100,100)[inner sep=0pt]  node (j4) {}; 
       \draw[->,black!40!green] (j1) to [bend left=80] (j2);
       \draw[->,black!40!green] (j1) to [bend left=80] (j3);
       \draw[->,black!40!green] (j1) to [bend left=80] (j4);
\draw[black!30!green,very thick] (40,100) -- (40+100/3,100) node[anchor=north]{};
\filldraw[red] (60, 210) circle(1pt)node[red] {};
\filldraw [blue] (60,100) circle (1pt) node[blue] {};
        
\end{axis}[]
\node[align=center, yshift=-1em] (title) 
    at (current bounding box.south)
    {Figure \ref{figc3}};

\end{tikzpicture}
\end{figure}
      \end{minipage}

\caption{A representation of Type~1, where $\lambda(x)=1$ (in blue, colors online), $a$ is a polynomial of second degree vanishing on the boundaries (in red), and the support of $\nu(x,\fdot)$ is contained in $[-x,1-x]$ (in green).}\label{figc1}
\caption{A representation of Type~2, where $\lambda$ has a pole of order 1 in $x=0$ (in blue), $a$ is a polynomial of second degree vanishing on the boundaries (in red), and the support of $\nu(x,\fdot)$ is contained in $[-x,0]$ (in green) for all $x\in E$. This in particular implies that the distance to the ``no-jump'' point always decreases if a jump occurs. Note that in $x=0$ there is no jump activity since $\lambda(0)=0$ and thus $\nu(0,E)=0$.}\label{figc2}
\caption{A representation of Type~3, where $\lambda$ has a pole of order 2 in $x^*\in(0,1)$ (in blue), $a$ is a polynomial of second degree on $E\setminus \{x^*\}$ vanishing on the boundaries (in red), and the support of $\nu(x,\fdot)$ is contained in $[-2(x-x^*),0]$, resp.~$[0,-2(x-x^*)]$, (in green) for all $x\in E$. This in particular implies that the distance to the ``no-jump'' point $x^*$ always decreases if a jump occurs. Note that in $x^*$ there is no jump activity since $\lambda(x^*)=0$, but there is an extra contribution to the diffusion coefficient at this point.}\label{figc3}
\end{center}
\end{figure}

\section{Examples of polynomial operators on the unit interval} \label{S:examples}

In this section we present a number of examples that illustrate 
the diverse behavior of polynomial jump-diffusions on $[0,1]$. While the diffusion case is simple -- 
the Jacobi diffusions (Type~0) are the only possibilities -- 
the complexity increases significantly in the presence of jumps. 
For instance, in Example~\ref{ex2} we obtain jump intensities with a countable number of poles in the state space.

\subsection{Examples with affine jump sizes}

\begin{example}
We start with a well-known example of a polynomial jump-diffusion on $[0,1]$; see \citet[Example 3.5]{Cuchiero/etal:2012}. Consider the Jacobi process, which is the solution of the stochastic differential equation
$$
\de X_t =\kappa_0(\theta_0-X_t)\de t+\sigma\sqrt{X_t(1-X_t)}\de W_t,\qquad X_0 =x_0\in [0,1],
$$
where $\theta_0 \in [0,1]$ and $\kappa_0,\sigma > 0$. This process can also be regarded as the unique solution to the martingale problem for $(\Gcal,\delta_{x_0})$, with the Type~0 operator
$$\Gcal f (x):=\frac 1 2 \sigma^2x(1-x) f''(x)+\kappa_0(\theta_0-x)f'(x).$$
This example can be extended by
adding jumps, where the jump times correspond to those of a Poisson process with intensity $\lambda$ and the jump size is a function of the process level. One can for instance specify that if a jump occurs, then the process is reflected in~$1/2$. In this case the process would be the unique solution to the martingale problem for $(\Gcal,\delta_{x_0})$, where
$$\Gcal f (x):=\frac 1 2 \sigma^2x(1-x) f''(x)+\kappa_0(\theta_0-x)f'(x)+\lambda\big(f(1-x)-f(x)\big),$$
which is an operator of Type~1 with $A=\sigma^2$, $\kappa= \kappa_0+2\lambda$, $\theta=\frac{\kappa_0\theta_0+\lambda}{\kappa_0+2\lambda}$, and $\mu=\lambda\delta_{(1,1)}$.
\end{example}

\begin{example}
The following example features a simple state-dependent jump distribution. Consider a L\'evy type operator $\Gcal$ whose jump kernel $\nu(x,\de\xi)$ is chosen such that $x+\xi$ is uniformly distributed on $(\alpha(x),\beta(x))$, where $\alpha,\beta\in\Pol_1(E)$ and $0\leq \alpha(x)\leq\beta(x)\leq1$ for all $x\in E$.
This in particular implies that $\alpha$ and $\beta$ can be written as 
$$\alpha(x)=\alpha_0(1-x)+\alpha_1x\qquad\text{and}\qquad \beta(x)=\beta_0(1-x)+\beta_1x$$
for some $0\leq\alpha_0\leq\beta_0\leq1$ and $0\leq\alpha_1\leq\beta_1\leq1$.
Choosing the drift coefficient $b$ suitably, the operator $\Gcal$ is then of Type~1 for $\mu$ being the pushforward of the uniform distribution on $(0,1)$ under the map $z\mapsto (1-z(\beta_1-\alpha_1)-\alpha_1, z(\beta_0-\alpha_0)+\alpha_0)$.

The solution to the corresponding martingale problem is a  Jacobi process extended by adding jumps, where the jump times correspond to those of a Poisson process with unit intensity, and the jump's target point is uniformly distributed on $(\alpha(x),\beta(x))$, given that the process is located at $x$ immediately before the jump. 
\end{example}

\begin{example}
Polynomial operators are not always easy to recognize at first sight. Consider a L\'evy type operator $\Gcal$ whose diffusion and drift coefficients $a$ and $b$ are zero, and whose jump kernel $\nu(x,\de\xi)$ is given by
$$
\nu(x,A)=\mathds1_{\{x\neq0\}} \frac {1-x} x\int_0^1\mathds 1_{A\setminus\{0\}}(-x\sin^2((x+ z)\pi)) \de  z.
$$
Despite the presence of the sine function, the operator $\Gcal$ satisfies all the conditions of Lemma~\ref{thm4}. It is thus polynomial and its martingale problem is well-posed. In fact, this operator is of Type~2. Using the periodicity of the sine function, one can show that $\nu(x,\de\xi)$ has affine jump sizes with $\lambda(x)=\frac {1-x} x \mathds 1_{\{x\neq0\}}$, $\gamma(x,y)=-xy$, and
$\mu$ being the pushforward of Lebesgue measure on $[0,1]$ under the map $z\mapsto\sin^2(z\pi)$. The associated polynomial jump-diffusion is a martingale since $b=0$. Moreover, the arrival intensity $\nu(x,E-x)$ of the jumps is given by $\frac{1-x} {x}\mathds1_{\{x\neq0\}}$, which is unbounded around zero.
\end{example}
\begin{example}
The Dunkl process with parameter ${n}\in\N_0$ is a polynomial jump-diffusion on $\R$, see e.g.~\citet[Example 3.7]{Cuchiero/etal:2012}, and can be characterized as the unique martingale whose absolute value is the Bessel process of dimension $1+2{n}$; see \cite{Gallardo:2006}. The corresponding polynomial operator $\Gcal^{\rm Dunkl}$ is of L\'evy type with diffusion and jump coefficients $a(x)=2+2{n}\mathds1_{\{x=0\}}$ and $b(x)=0$, and jump kernel
$$\nu(x,\de\xi)=\mathds1_{\{x\neq0\}}\frac {n} {2x^2}\delta_{-2x}(\de\xi).$$
 The arrival intensity of its jumps is thus given by $\nu(x,\R)=\frac {n}{2x^2}\mathds1_{\{x\neq0\}}$, which is a rational function with a pole of second order in $x=0$. 

Observe that $\nu(x,\de\xi)$ exhibits several similarities with jump kernels of operators of Type~3, such as the form of the arrival intensity of the jumps, and the extra contribution to the diffusion coefficient at the ``no-jump'' point $x=0$. In fact, defining $\tilde f\!:=\!f(\cdot+\frac12)$ and
$$
\Gcal f(x)=x(1-x)\Gcal^{\rm Dunkl} \tilde f(x-1/2),
$$
we obtain a polynomial operator of Type~3 with ``no-jump'' point $x^*=1/2$.
\end{example}

\subsection{Constructions using conic combinations}

We provide two examples illustrating the usefulness of Theorem~\ref{lem3} for combining operators with affine jump sizes to achieve specifications with interesting properties.


\begin{example}\label{ex2}
We now construct a polynomial operator whose martingale problem is well-posed, such that the arrival intensity of the jumps is unbounded around infinitely many points, but finite for all $x\neq1/2$.

Let $\Gcal_n$, $n\geq3$, be operators of Type~3 with ``no-jump'' points $x_n^*=\frac 1 2+ \frac 1 n$. Let their diffusion coefficients be given by 
$$a_n(x)=\frac 1{3n^2}x_n^*(1-x_n^*)\mathds 1_{\{x=x_n^*\}},$$
the drift coefficients be 0, and the parameters of the jump kernels $\nu_n(x,\de\xi)$ be given by
$$\lambda_n(x)=n^{-2}\frac{x(1-x)}{(x-x_n^*)^2}\mathds1_{\{x\neq x_n^*\}},\quad \gamma_n(x,y)=-y(x-x_n^*),$$
and $\mu$ be Lebesgue measure on $[0,1]$. Note that for all $k\geq2$ we have
\begin{equation}\label{eqn57}
\sum_{n=3}^\infty\left( a_n(x)\delta_{k2} +\int\xi^k \nu_n(x,\de\xi)\right)=\frac{x(1-x)}{k+1}\sum_{n=3}^\infty n^{-2}(x^*_n-x)^{k-2}<\infty.
\end{equation}
By Theorem~\ref{lem3} and Lemma~\ref{lem7} this implies that the operator $\Gcal:=\sum_{n=3}^\infty \Gcal_n$ is again polynomial  and its martingale problem is well-posed. In particular, $\Gcal$ is a L\'evy type operator with coefficients $a(x)=\sum_{n=3}^\infty a_n(x)$ and $b(x)=0$, and jump kernel
$\nu(x,\fdot):=\sum_{n=3}^\infty \nu_n(x,\fdot)$. As a result, the arrival intensity of the jumps is given by
$$\nu\big(x,E-x\big)=\sum_{n=3}^\infty \lambda_n(x)=x(1-x)\sum_{n=3}^\infty\frac 1 {n^2(x-x^*_n)^2}\mathds1_{\{x\neq x^*_n\}},$$
which is unbounded around each $x_n^*$ but finite for all $x\neq1/2$. At $x=1/2$ the jump intensity is infinite. Figure~\ref{fig2} contains an illustration.
\end{example}

\begin{figure}[h!]
\begin{center}
\includegraphics[scale=0.6]{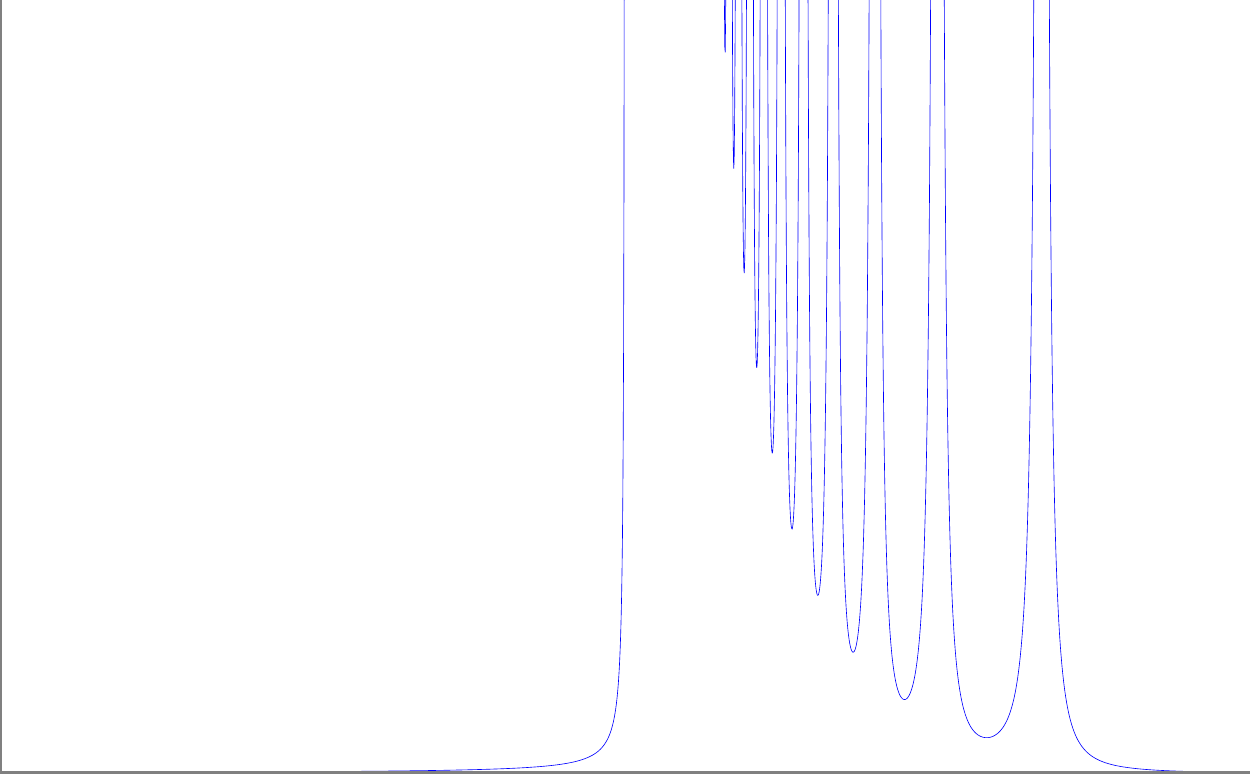}
\end{center}
\caption{A graphical representation of arrival intensity of the jumps $\nu(x, E-x)$ appearing in Example~\ref{ex2}.}\label{fig2}
\end{figure}

\begin{example}
This example shows that the operator of a polynomial diffusion, or equivalently an operator of Type~0, can always be written as the limit of ``pure jump'' polynomial operators, i.e.~with zero diffusion coefficients. Consider the Jacobi diffusion with operator $\Gcal$ given by
$$\Gcal f(x):=Ax(1-x)f''(x)+\kappa(\theta-x)f'(x),$$
for some $A\in\R_+$, $\kappa\in\R_+$, and $\theta\in[0,1]$. Let then $\Gcal_n$ be an operator of Type~2 and suppose that its diffusion coefficient $a_n$ is zero, the drift coefficient is given by $b_n(x)=\kappa(\theta-x)$, and the parameters of the jump kernel $\nu_n(x,\de\xi)$ are
$$\lambda_n(x)=n^2\frac{A(1-x)}{x}\mathds1_{\{x\neq0\}},\quad \gamma_n(x,y)=-yx,\quad\mu=\delta_{1/n}.$$
Observe that, trivially, we have $\lim_{n\to\infty}b_n(x)=\kappa(\theta-x)$. Also,
\begin{gather*}
\lim_{n\to\infty}\left(a_n(x)+\int\xi^2\nu_n(x,\de\xi)\right)=Ax(1-x) \quad\text{and}\quad \lim_{n\to\infty}\int\xi^k\nu_n(x,\de\xi)=0, \quad k\ge 3.
\end{gather*}
By Lemma~\ref{lem7} we thus conclude that $\Gcal= \lim_{n\to\infty}\Gcal_n$ in sense of Theorem~\ref{lem3}.
\end{example}

\subsection{Mixed affine jump sizes}
Consider now a L\'evy type polynomial operator $\Gcal$ whose jump kernel has mixed affine jump sizes
in the sense of Section \ref{sec3}, i.e.,
\begin{equation}\label{eqn56}
\nu(x,\de\xi)=\sum_{\ell=1}^L\nu_\ell(x,\de\xi),
\end{equation}
where each kernel $\nu_\ell(x,\de\xi)$ has affine jump sizes. Suppose the martingale problem for $\Gcal$ is well-posed, or equivalently, its triplet satisfies the conditions of Lemma~\ref{thm4}. A natural question is now whether the individual kernels $\nu_\ell(x,\de\xi)$ also satisfy the conditions of Lemma~\ref{thm4}. If this were to be true, it would have the pleasant consequence that $\Gcal$ could be represented as a sum of operators of Types~0--4. 
Unfortunately this is not the case, which we illustrate in Example~\ref{ooc} below. In fact, there exist kernels of the form~\eqref{eqn56} that cannot even be obtained as an infinite conic combination of the kernels appearing in Types~0--4.

\begin{example}\label{ooc}
Consider a L\'evy type operator $\Gcal$, whose coefficients are given by $a(x)=0$, $b(x)=1-2x$, and whose jump kernel is given by (\ref{eqn56}) for $L=2$, where $\nu_1(x,\de\xi)$ and $\nu_2(x,\de\xi)$ have affine jump sizes with parameters $\lambda_1(x)=\frac{1}{x(x+1)}\mathds1_{\{x\neq0\}}$, $\mu_1=\delta_{(1,0)}$, and  $\lambda_2(x)=\frac{2}{(1-x)(x+1)}\mathds1_{\{x\neq1\}}$, $\mu_2=\delta_{(0,1/2)}$. 
Observe that 
$$\gamma(x,y)=-x\quad\mu_1\text{-a.s.\qquad and\qquad }  \gamma(x,y)=\frac 1 2(1-x)\quad\mu_2\text{-a.s.}$$
One can verify that $\Gcal$ satisfies all the conditions of Lemma~\ref{thm4}, and is thus polynomial and its martingale problem is well posed.

Assume now for contradiction that $\nu(x,\de\xi)=\sum_{\ell=1}^\infty\tilde\nu_\ell(x,\de\xi)$ for some kernels $\tilde\nu_\ell(x,\de\xi)$ that satisfy the conditions of Lemma~\ref{thm4} for some coefficients $a_\ell(x)$ and $b_\ell(x)$. By Theorem~\ref{thm1}, each $\tilde\nu_\ell(x,\de\xi)$ then follows one of Types~0--4. Let $\tilde\lambda_\ell(x)$ and $\tilde\mu_\ell(x)$ be the parameters of the jump kernel $\tilde\nu_\ell(x,\de\xi)$.

Since $\supp\,\nu(x,\fdot)\subseteq\{-x,(1-x)/2\}$, we also have $\supp\, \tilde\nu_\ell(x,\fdot) \subseteq\{-x,(1-x)/2\}$ for all $x\in E$, or equivalently,
\begin{equation}\label{eqn15}
\tilde\mu_\ell=\alpha_\ell\ \delta_{(1,0)}+\beta_\ell\ \delta_{(0,1/2)}
\end{equation}
for some $\alpha_\ell,\beta_\ell\geq0$. This already excludes that $\tilde\nu_\ell(x,\de\xi)$ is of Type~3 or~4, and gives us that for all $x\in(0,1)$
$$\tilde\lambda_\ell(x)=\begin{cases}
\frac{q^\alpha_{\ell}(x)}{x}&\text{if $\beta_\ell=0$,}\\
\frac{q^\beta_{\ell}(x)}{1-x}&\text{if $\alpha_\ell=0$,}\\
c_\ell &\text{otherwise,}
\end{cases}$$
for some $q^\alpha_{\ell}, q^\beta_{\ell}\in\Pol_1(E)$ and $c_\ell\in\R_+$.
In particular note that for all $x\in (0,1)$ and $\ell\in\N$,
\begin{equation}\label{eqn17}
\alpha_\ell\tilde\lambda_\ell(x)=\alpha_\ell\frac{q^\alpha_{\ell}(x)}{x}\qquad\textup{and}\qquad\beta_\ell\tilde\lambda_\ell(x)=\beta_\ell\frac{q^\beta_{\ell}(x)}{1-x}
\end{equation}
and hence, since $\Pol_1(E)$ is closed under pointwise convergence,
\begin{equation*}
\sum_{\ell=1}^{\infty}\int\xi^n\tilde\nu_\ell(x,\de\xi)=\frac{q^\alpha(x)}x(-x)^n+\frac{q^\beta(x)}{1-x}\bigg(\frac{1-x}2\bigg)^{n},
\end{equation*}
for all $n\in\N$ and some $q^\alpha,q^\beta\in\Pol_1(E)$.
Since $\int\xi^n\nu(x,\de\xi)=\sum_{\ell=1}^{\infty}\int\xi^n\tilde\nu_\ell(x,\de\xi)$ by assumption, we obtain
$$
 \frac{-(-x)^{n-1}+((1-x)/2)^{n-1}}{x+1}=-q^\alpha(x)(-x)^{n-1}+\frac 1 2 q^\beta(x) \bigg(\frac{1-x}2\bigg)^{n-1},
$$
for all $x\in (0,1)$, $n\in\N$. The shortest way to see that this condition cannot be satisfied is to use that if two polynomials coincide on $(0,1)$ they have to coincide on $\R$, too. But, choosing $x=-1$ we obtain
$$-q^\alpha(-1)+\frac 1 2q^\beta(-1)=\frac{n-1} 2$$
for all $n\in\N$, which is clearly not possible.
\end{example}


\begin{example}\label{ex1}
It is possible to show that operators with jump kernels of the form \eqref{eqn56} can have intensities $\lambda_\ell$ with multiple poles of multiple order outside the state space. On the other hand, under some non-degeneracy conditions, they can only have a single pole of order at most 2 inside the state space. 
We develop this idea in more detail for the case when $\nu(x,\fdot)$ consists of finitely many atoms for all $x\in E$. 

Let $\Gcal:\Pol(E)\to C(E)$ be an operator of the form described in Lemma~\ref{thm4} and suppose that its jump kernel $\nu(x,\de\xi)$ is supported on $\{\gamma_1(x),\ldots,\gamma_L(x)\}$, where $\gamma_\ell\in\Pol_1(E)$, $\ell=1,\ldots,L$, are pairwise distinct polynomials with $x+\gamma_\ell(x)\in E$ for all $x\in E$. As a result, we have
\begin{equation}\label{eqn49}
\nu(x,\de\xi)=\sum_{\ell=1}^L\lambda_\ell(x)\delta_{\gamma_\ell(x)}(\de\xi)
\end{equation}
for some functions $\lambda_\ell:E\to\R_+$. For $n\geq2$, set $r_n:=\int\xi^n\nu(\fdot,\de\xi)=\sum_{\ell=1}^L\lambda_\ell\gamma_\ell^n$,
and recall that $r_n\in\Pol_n(E)$ for all $n\geq3$ and that $r_2$ is bounded on $E$.
Using the nonnegativity of $\lambda$ and the boundary conditions for $a$, one can then establish the following properties, which we state here without proof.
\begin{enumerate}
\item If $\lambda_\ell$ has a pole at a point $x_0\in E$, then $\gamma_\ell(x_0)=0$. Moreover in this case, analogously to Types 2 and 3, if $x_0\in\{0,1\}$, the order of the pole is 1 and if $x_0\in(0,1)$, the order of the pole is 2. Note that nonnegativity of $\lambda_\ell$ and the fact that $\gamma_\ell\in\Pol_1(E)$ imply that $\lambda_\ell$ can have a pole in at most  one point of the state space.
\item $r_2\in\Pol_2(E\setminus\{x_1^*,\ldots,x^*_L\})$, where $x^*_\ell$ denotes the zero of $\gamma_\ell$, and we have
\begin{equation}\label{eqn31}
\lambda_\ell=\frac {q_\ell} {\gamma_\ell^2\prod_{j\neq \ell}(\gamma_\ell-\gamma_j)}\mathds1_{\{\gamma_\ell\neq0\}},
\end{equation}
where $q_\ell\in\Pol_{L+1}(E)$ and on $E\setminus\{x^*_1,\ldots,x_L^*\}$ it is given by
$$q_\ell=\sum_{k=0}^{L-1}\bigg((-1)^{k} r_{L-k+1}\sum_{\substack{\ell_1<\ldots<\ell_{k}\\ \ell_1,\ldots,\ell_k\neq \ell}}\gamma_{\ell_1}\cdots\gamma_{\ell_{k}}\bigg).$$
\item Since $a+\int\xi^2\nu(\fdot,\de\xi)\in\Pol_2(E)$ by Lemma~\ref{thm4}, we can conclude that
\begin{equation}\label{eqn50}
a(x)=Ax(1-x)+\sum_{\ell=1}^L\bigg(\frac {(-1)^{L-1}q_\ell(x)} {\prod_{j\neq \ell}\gamma_j(x)}\mathds1_{\{x=x^*_\ell\}}\bigg),
\end{equation}
for some $A\in\R_+$ and all $x\in E$.
\end{enumerate}

Conversely, fix a sequence of polynomials $r_{k+2}$, $k=0,\ldots,L-1$, such that $r_{k+2}\in\Pol_{k+2}(E)$ for all $k$. If for some affine functions $\gamma_1,\ldots,\gamma_L$ as above, the functions $\lambda_\ell$ given by equation \eqref{eqn31} satisfy (i) and are all nonnegative on $E$, one can conclude that for $\nu(x,\de\xi)$ as in \eqref{eqn49}, $a$ as in \eqref{eqn50}, and a suitably chosen $b\in\Pol_1(E)$, the corresponding L\'evy type operator is polynomial and its martingale problem is well-posed.
\end{example}

\begin{remark}
It is interesting to observe that Shur polynomials appear naturally in the context of Example \ref{ex1}. Indeed, by point (ii) we know that each $\lambda_\ell(x)$, and thus every moment $r_{n}(x)$ of the measures $\nu(x, \fdot)$, is uniquely determined by $\gamma_1,\ldots,\gamma_L$ and $r_2,\ldots,r_{L+1}$. More precisely for all $n>L+1$ we can write
$$r_n=\sum_{k=1}^L(-1)^{L-k} s_{\mu_{L,n,k}}(\gamma_1,\ldots,\gamma_L) r_{k+1},$$
where $\mu_{L,n,k}=(\mu_{L,n,k}^1,\ldots,\mu_{L,n,k}^L)$ is the partition given by
$$\mu_{L,n,k}^1=n-L-1,\quad \mu_{L,n,k}^2=\ldots= \mu_{L,n,k}^{L-k+1}=1,\quad\text{and}\quad\mu_{L,n,k}^{L-k+2}=\ldots=\mu_{L,n,k}^L=0,$$
and $s_{\mu_{L,n,k}}$ is the corresponding Shur polynomial.
\end{remark}

We now propose two interesting applications of Example~\ref{ex1}, showing that it can happen that $\lambda_1,\ldots,\lambda_L$ have poles of high order and in several points outside the state space.

\begin{example}
Consider a kernel of the form described in (\ref{eqn49}) for
$$\gamma_1(x)=-x,\quad \gamma_2(x)=1-x, \quad\gamma_3(x)=\frac 1 3 ( 1  -2x),\quad\gamma_4(x)=\frac 2 3 ( 1  -2x).$$
Defining $\lambda_\ell$ through expression (\ref{eqn31}) where we set
$$r_2(x)=1,\quad \!r_3(x)=\frac{1-2x}2,\quad\! r_4(x)=\frac{2x^2-2x+5}{18},\quad \! r_5(x)=\frac{(2x-1)(5x^2-5x+1)}6,$$
we obtain
\begin{align*}
\lambda_1(x)&=\frac 9 2 \frac{(1-x)}{x(x+1)(2-x)},\qquad\lambda_2(x)=\frac 9 2 \frac{x}{(1-x)(x+1)(2-x)},\\
\lambda_3(x)&=\frac{9}{(x+1)(2-x)},\qquad \quad\ \lambda_4(x)=\frac 9 4 \frac{1}{(x+1)(2-x)},
\end{align*}
for all $x\in E$. Note that the rational functions $\lambda_\ell$ satisfy point (i) of Example~\ref{ex1} and are all nonnegative on $E$. As a result, choosing the diffusion and drift coefficients suitably, $\Gcal$ is a polynomial operator whose martingale problem is well-posed. Observe that each $\lambda_\ell$ has a pole in $x=-1$ and $x=2$.
\end{example}

\begin{example}
Consider a kernel of the form described in (\ref{eqn49}) for
$$\gamma_1(x)=-x,\quad \gamma_2(x)=\frac 1 2 (1-x), \quad\gamma_3(x)=\frac 1 3 ( 1  -2x).$$
Defining $\lambda_\ell$ through expression (\ref{eqn31}) where we set
$$r_2(x)=1,\quad r_3(x)=\frac{1-2x}2,\quad r_4(x)=\frac{10x^2-9x+3}{12},$$
we obtain
$$\lambda_1(x)= \frac{1}{x(x+1)^2},\quad\lambda_2(x)= \frac{4(2x+1)}{(1-x)(x+1)^2},\quad\lambda_3(x)=\frac{27 x^2}{(1-2x)^2(x+1)^2},$$
for all $x\in E$.  Note that the rational functions $\lambda_\ell$ satisfy point (i) of Example~\ref{ex1} and are all nonnegative on $E$. As a result, choosing the diffusion and drift coefficients suitably, $\Gcal$ is a polynomial operator whose martingale problem is well-posed. Observe that each $\lambda_\ell$ has a pole of second order in $x=-1$.
\end{example}

\section{The unit simplex}\label{sec6}

Throughout this section the state space $E\subset\R^d$ is the unit simplex of dimension $d-1$, which we denote by
$$
E:=\Delta^d=\left\{x \in \R_+^d \, : \, \sum_{i=1}^d x_i=1\right\}.
$$
Similarly as in Section~\ref{sec4} our goal is to provide a characterization of polynomial 
jump-diffusions on $E$ with affine jump sizes. 
Again, we combine Theorem~\ref{lem2} and Lemma~\ref{lem1} to specialize Theorem \ref{thm3} to the state space $E$. 
The proof is given in Section~\ref{app22}.

\begin{lemma}\label{thm6}
A linear operator $\Gcal:\Pol(E)\to C(E)$ is polynomial and its martingale problem is well-posed if and only if it is of form \eqref{eq:G LK} and the corresponding triplet $(a,b,\nu)$ satisfies
\begin{enumerate}
\item $a(x)\in\mathbb S_+^d$ for all $x\in E$ and $\nu(x,\de\xi)$ satisfies \eqref{thenu},
\item $a_{ii}(x)=0$ and $b_i(x)-\int\xi_i\nu(x,\de\xi)\geq0$ for all $x\in E\cap\{x_i=0\}$,
\item $a\mathbf 1=0$ and $b^\top\mathbf1=0$,
\item $b_i\in\Pol_1(E)$,
 $a_{ij}+\int\xi_i\xi_j\nu(\fdot,\de\xi)\in\Pol_2(E)$, and
$\int\xi^{\mathbf k}\nu(\fdot,\de\xi)\in\Pol_{|\mathbf k|}(E)$ for all $|\mathbf k|\geq3$.
\end{enumerate}
\end{lemma}
Observe that conditions (i) and (iii) guarantee that $\Gcal$ is of L\'evy Type. 
This in particular ensures that the right-hand side of \eqref{eq:G LK} can be computed using an arbitrary representative.

\begin{remark}
Condition (ii) implies that $\int|\xi_i|\nu(x,\de\xi)<\infty$ for all $x\in E\cap\{x_i=0\}$. 
Analogously to the unit interval case, this give us some intuition 
about the behavior of the solution $X$ on the boundary segment $x\in E\cap\{x_i=0\}$. Indeed, even if the component orthogonal to the boundary of the purely discontinuous martingale part of $X$ is necessarily of finite variation, the other components do not need to satisfy this property. 
Moreover, since $a(x)\in\mathbb S_+^d$, condition (ii) also implies that $a_{ij}(x)=0$ for all $j\in\{1,\ldots,d\}$.
\end{remark}

We now focus on the setting of affine jump sizes in the sense of Definition~\ref{D:AJS}. We thus consider L\'evy type operators $\Gcal$ of the form
\begin{equation} \label{eqn51}
\begin{aligned}
\Gcal f(x) &= \frac{1}{2}\tr\left( a(x) \nabla^2 f(x) \right) + b(x)^\top \nabla f(x) \\
&\qquad + \lambda(x)\int \left( f(x + \gamma(x,y)) - f(x) - \gamma(x,y)^\top \nabla f(x) \right) \nu(x,\de\xi),\\
\end{aligned}
\end{equation}
where $\lambda$ is nonnegative and measurable, and $\gamma(x,y)$ is affine in $x$. In order to describe the form of the jump sizes, let us introduce the set $(\Delta^d)^d$ which is given by
$$
(\Delta^d)^d=\{y=(y^1,\ldots,y^d) \in \mathbb{R}_+^{d \times d} \, :\, y^i \in \Delta^d \textrm{ for all } i \in \{1,\ldots,d\} \}.
$$
\paragraph{Type~0.} 
For some $\alpha_{ij}\in\R_+$, $\alpha_{ij}=\alpha_{ji}$, $B \in \mathbb{R}^{d\times d}$ such that $B_{ij}\geq0$ for $i\neq j$ and $B_{ii}=-\sum_{j\neq i}B_{ji}$, let
\begin{align}
a_{ii}(x)&=\sum_{i\neq j}\alpha_{ij}x_ix_j\quad\text{and}\quad a_{ij}(x)=-\alpha_{ij}x_ix_j \quad\textrm{ for all }  i\neq j,\label{eq:forma}\\
b(x)&=Bx,\label{eq:formb}
\end{align}
and set $\lambda=0$.
Then $\mathcal{G}$ is a polynomial operator whose martingale problem 
is well-posed. 
The solutions $X$ are multivariate Jacobi-type diffusion processes 
which have been characterized in this form by
\citet[Section 6.3]{Filipovic/Larsson:2016}. 
In the special case where $\alpha_{ij}=\sigma^2$ for all $i,j$, they correspond to Wright-Fisher diffusions, 
which are also known under the name multivariate Jacobi process; see \citet{GJ:06}.

\paragraph{Type~1.} 
Let $\lambda(x) = 1$ and $a(x), b(x)$ be given by~\eqref{eq:forma} and~\eqref{eq:formb}.
For all $y \in (\Delta^d)^d$ set
\begin{align}\label{eq:gammaspec}
\gamma(x,y)=\sum_{i=1}^d (y^i-e_i) x_i, 
\end{align}
and let $\mu$ be a nonzero measure on $(\Delta^d)^d$. 
If the boundary conditions
$$
B_{ij}-\int y^j_i \mu({\de y}) \geq 0
$$
hold for all $i\neq j$, then $\mathcal{G}$ is a polynomial operator whose martingale problem is well-posed.

Note that the boundary conditions imply that
$$\int|\xi|\nu(x,\de\xi)\leq \sum_{i=1}^d\int |y^i-e_i|\mu(\de y)$$
is bounded. 
Hence the resulting process behaves like a multivariate Jacobi-type diffusion process in the spirit of~\citet[Section 6.3]{Filipovic/Larsson:2016}, generalized to include summable jumps. The arrival intensity of the jumps is $\nu(x,E-x)=\mu(\{y: \gamma(x,y)\ne0\})$, which may or may not be finite.

\paragraph{Type~2.} 
Fix $i \in \{1, \ldots, d \}$. Let $a(x), b(x)$ be given by~\eqref{eq:forma} and~\eqref{eq:formb}, and let $\lambda(x) = \frac{q_1(x)}{x_i}\mathds1_{\{ x_i\neq 0\}}$ for some nonnegative $q_1 \in \Pol_1(E)$ such that $\lambda$ is not constant on $E\cap\{x_i\neq0\}$.
Furthermore, for $y \in \Delta^d$ we define
\[
\gamma(x,y)=(y-e_i)x_i,
\]
and let $\mu$ be a nonzero square-integrable measure on $\Delta^d\setminus\{e_i\}$.
If the boundary conditions
\begin{equation} \label{eq:simplex 2 bdry}
B_{kj}-q_1(e_j)\int y_k \mu({\de y})  \geq 0
\end{equation}
hold for all $k\neq i$ and $j\neq k$, then $\mathcal{G}$ is a polynomial operator whose martingale problem is well-posed.

If $q_1(x)=Lx_k$ for some $k\neq i$ and $L>0$, the jumps need not to be summable. More precisely, we can have $\int |y_i-1| \mu(dy)=\int |y_k| \mu(dy) =\infty$.
Otherwise, if $q_1(x)$ is not proportional to $x_k$ on $E$ for any $k$, the expression
$\int|\xi|\nu(x,\de\xi)$
is bounded, and the solution $X$  to the martingale problem for $\Gcal$ has thus summable jumps.  Indeed, 
$$\int|\xi_k|\nu(x,\de\xi)\leq \sup_{x\in E}q_1(x)\int y_k\mu(\de y),$$
which is bounded due to~\eqref{eq:simplex 2 bdry} and the existence of some $x\in E\cap\{x_k=0\}\cap\{x_i\neq0\}$ such that $q_1(x)\neq0$; see Lemma \ref{clem1} in Section \ref{app22} for more details on the second point.

The arrival intensity of the jumps is $\nu(x,E-x)=\lambda(x)\mu( \Delta^d \setminus \{e_i\})$ and hence, even if $\mu$ is a finite measure, the jump intensity is unbounded around $x_i=0$. Moreover, due to the form of $\gamma(x,y)$, $X$ can only jump in the direction of the boundary segment  $E\cap\{x_i=0\}$, and since $\nu(x,E)=0$ whenever $x_i=0$, $X$ cannot leave this boundary segment by means of a jump. Figure~\ref{figcc2} illustrates the form of $\lambda$ and the support of $\gamma(x,y)$ under $\mu$.

\paragraph{Type~3.} 
Let $i,j \in\{1,\ldots, d\}$ be such that $i\neq j$, and fix some constant $c >0$.
Consider the hyperplane $\{cx_i=x_j\}$ which will be a ``no-jump'' region. Let $b$ be given by~\eqref{eq:formb} and set
\[
\lambda(x)=\frac{q_2(x)}{(-cx_i+x_j)^2}\mathds1_{\{cx_i\neq x_j\}}
\] 
for some $q_2 \in \Pol_2(E)$ given by $q_2(x)=\sum_{k=1}^d(q_{ik}x_ix_k+q_{jk}x_jx_k)$, 
where $q_{ k\ell}\in\R$ are chosen such that $\lambda$ is nonnegative, and nonconstant on $\{cx_i\neq x_j\}$. Furthermore, define 
\[
\gamma(x,y)=y(-cx_i+x_j)(e_i-e_j)
\] 
and let $\mu$ be a nonzero square-integrable measure on $\big(0,\frac{1}{c} \wedge 1\big]$. Finally, let 
\[
a(x)=a^c(x)+a^\nu(x)A^\nu\mathds 1_{\{cx_i=x_j\}}
\] 
where $a^c$ is of form~\eqref{eq:forma}, $A^\nu\in\R^{d\times d}$ is a symmetric matrix given by $A^\nu_{ii}=A^\nu_{jj}=1$, $A^\nu_{ij}=-1$, and $A^\nu_{k\ell}=0$ if $k\notin\{i,j\}$, and where
$$a^\nu(x)=q_2(x)\int y^2\mu(\de y).$$
If  the boundary conditions
\begin{equation}\label{eq:driftpos}
\sum_{\ell\neq k} B_{k\ell}x_\ell-\frac{q_2(x)}{(-cx_i+x_j)}\mathds1_{\{ cx_i\neq x_j\}}\int y(\delta_{ik}-\delta_{jk}) \mu({\de y})  \geq 0
\end{equation}
are satisfied for all  $x \in E \cap\{x_k=0\}$ and $k\in\{1,\ldots,d\}$, then $\Gcal$ is a polynomial operator whose martingale problem is well-posed. Note in particular that for $k\notin\{i,j\}$, the condition (\ref{eq:driftpos}) coincides with $b_k\geq0$ on $E\cap\{x_k=0\}$.

If the numerator of $\lambda$ is of the form $q_2(x)= 2q_{ij}x_ix_j$ for some $q_{ij}\in\R_+$, the solution $X$ to the martingale problem for $\Gcal$ may have nonsummable jumps.  If $q_2(x)$ is not of this form, then by similar reasoning as for Type~2, the boundary conditions imply that $\int y\mu({\de y}) < \infty$ and thus $X$ has summable jumps.
The arrival intensity of the jumps is
$$\nu(x,E-x)=\lambda(x)\mu\left(\left(0,\frac{1}{c} \wedge 1\right]\right).$$
 As a result, even if $\mu$ is a finite measure, the jump intensity has a singularity of order two along $\{cx_i=x_j\}$, which results in a contribution of $a^\nu(x)A^\nu$ to the diffusion coefficient. Moreover, due to the form of $\gamma(x,y)$, the jumps of $X$ are always in the direction of the ``no-jump'' hyperplane $\{cx_i=x_j\}$. Although the jumps may overshoot $\{cx_i=x_j\}$, they always serve to reduce the distance to $ \{cx_i=x_j\}$. In particular, since $\nu(x,E-x)=0$ for all $x \in  E\cap\{cx_i=x_j\}$, $X$ cannot leave  $\{cx_i=x_j\}$ by means of a jump. Figure~\ref{figcc3} illustrates the form of $\lambda$ and the support of $\gamma(x,y)$ under $\mu$.\\


\begin{figure}[h!]
\begin{center}
\includegraphics[scale=1]{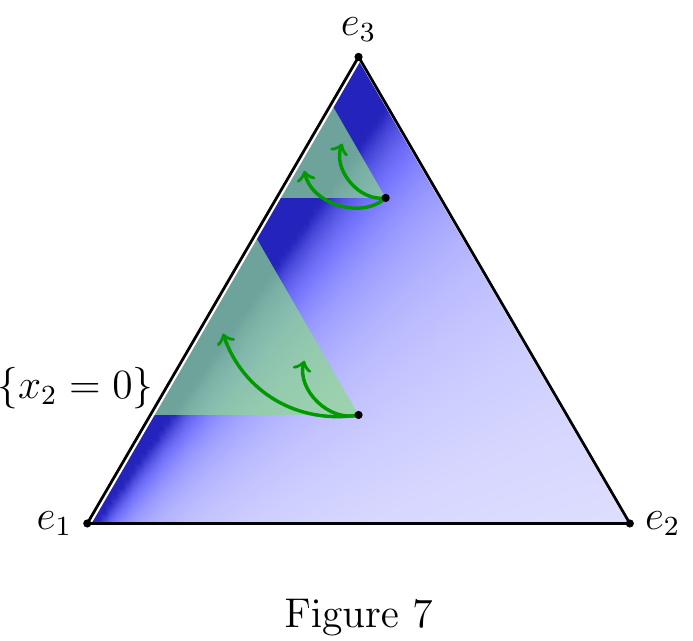}\qquad \includegraphics[scale=1]{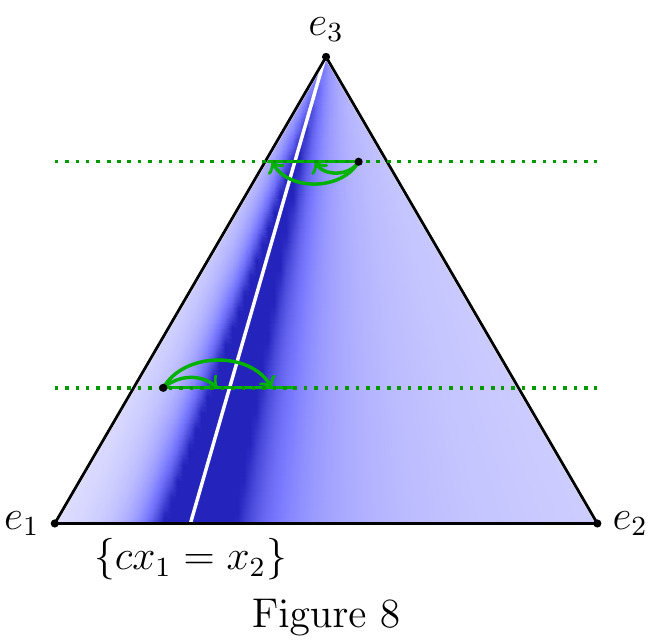}
\caption{A representation of Type~2, where $\lambda$ explodes on the boundary segment $\{x_2=0\}$ (in blue, colors online) and the support of $\nu(x,\fdot)$ is always contained in $\big(\Delta^3-e_2\big)x_2$ (in green) for all $x\in E$. This in particular implies that  the distance to the ``no-jump'' hyperplane $\{x_2=0\}$ always decreases if a jump occurs. Note that on $\{x_2=0\}$ there is no jump activity since $\nu(x,\de\xi)=0$ for all $x\in E\cap\{x_2=0\}$.}\label{figcc2}
\caption{A representation of Type~3, where $\lambda$ explodes on the hyperplane $\{c x_1=x_2\}$ (in blue) and the support of $\nu(x,\fdot)$ is always contained in $\Delta^3\cap\{x_3=0\}$ (in green) for all $x\in E$. Moreover the distance to the ``no-jump'' hyperplane $\{c x_1=x_2\}$ always decreases if a jump occurs. Note that on $\{c x_1=x_2\}$ there is no jump activity since $\nu(x,\de\xi)=0$ for all $x\in E\cap\{c x_1=x_2\}$, but  we know form the description of this type that there is an extra contribution to the diffusion matrix $a$.}\label{figcc3}
\end{center}
\end{figure}

In order to simplify the analysis, in particular in view of the arguments outlined in Remark \ref{rem:nonexistence},
we do not consider operators corresponding to Type~4 on the unit simplex. 
A condition on the jump kernel excluding this class is given by the following assumption.

\begin{assumption}\label{A}
The condition $\lambda\gamma_i(\fdot,y)^{3} \in \Pol_{3}(E)$ holds for all $i \in \{1,\ldots,d\}$ and all $y \in \supp(\mu)$.
\end{assumption}

The polynomial property of $\Gcal$ implies that the integrated quantities $\lambda \int \gamma_i(\fdot,y)^3 \mu(dy)$ lie in $\Pol_{3}(E)$. Assumption \ref{A} strengthens this by requiring that the functions $\lambda\gamma_i(\fdot,y)^{3}$ themselves lie in $\Pol_{3}(E)$. This is a natural assumption, in particular in view of Types 1-3 on the unit interval.
Moreover, it will turn out in the course of the proof of Theorem~\ref{th:mainsimplex} that Assumption \ref{A} implies under the condition 
of affine jump sizes and nonconstant $\lambda$
that $$\gamma(\fdot,y)=H(y)P_1$$ where $H$ is a $\mu$-measurable function and $P_1 \in \Pol_1(E)$.
Analogous to the unit interval, the ``no-jump'' region is the intersection of $E$ with the hyperplane given by the zero set of a polynomial of first degree. The following theorem states the announced characterization of polynomial jump-diffusions with polynomial jump sizes under Assumption~\ref{A}. The proof is given in Section~\ref{app22}.

\begin{theorem}\label{th:mainsimplex}
Let $\Gcal$ be a polynomial operator whose martingale problem is well-posed. If the associated jump kernel has affine jump sizes and satisfies Assumption~\ref{A}, then $\Gcal$ necessarily belongs to one of the Types~0-3.
\end{theorem}

\section{Applications} \label{S:appl}

In this section we outline two natural applications in finance of polynomial jump-diffusions on the unit simplex. The first application concerns stochastic portfolio theory, while the second application is in the area of default risk.

\subsection{Market weights with jumps in stochastic portfolio theory} \label{S:SPT_appl}

In the context of stochastic portfolio theory (SPT), polynomial diffusion models for the process of market weights have been found capable of matching certain empirically observed properties when calibrated to \emph{jump-cleaned} data; see~\cite{C:16, CGGPST:16}. This concerns the typical shape and dynamics of the capital distribution curves, but also features such as high volatility for low capitalized stocks. As mentioned in the introduction, a crucial deficiency of these models is the lack of jumps since they are present in typical market data; see Figure~\ref{fig:2}.

We now demonstrate how the results of Section~\ref{sec6} can be used to construct polynomial \emph{jump}-diffusion models for the market weights. We focus on a concrete specification that extends the \emph{volatility stabilized models} introduced by \cite{FK:05} by including jumps of Type 2. In the standard (diffusive) volatility stabilized model, the market weights follow a Wright-Fisher diffusion, which is a special case of Type 0 with parameters
\[
\alpha_{ij}=1 \quad \forall i\neq j \quad \text{and} \quad B=\frac{1+\beta}{2}\mathbf{1}\mathbf{1}^{\top}- \frac{d(1+\beta)}{2} \text{Id},
\]
for some $\beta \geq 0$. These models have two key properties which are of particular relevance in SPT: First, the market weights remain a.s.~in the relative interior of $\Delta^d$, denoted by $\mathring{\Delta}^d$. Second, the model allows for relative arbitrage opportunities.  We may preserve these features by adding jumps of Type 2. More precisely, we consider a model for the market weights $(X_t)_{t \geq 0}$  of the form
\begin{align*}
X_{t,i}&=\int_0^t \left(\frac{1+\beta}{2}-\frac{d(1+\beta)}{2}X_{s,i}\right)ds + \int_0^t \sqrt{X_{s,i}}(1-X_{s,i})dW_{s,i}-\sum_{i\neq j}X_{s,i}\sqrt{X_{s,i}}dW_{s,j}\\
 &\quad+\int_0^t \int \xi_i (\mu^{X}(d\xi,ds)-\nu(X_s,d\xi)ds),
\end{align*}
where $\mu^X$ denotes the integer-valued random measure associated to the jumps of $X$, and $W$ a $d$-dimensional standard Brownian motion. The jump specification is given as a sum of Type 2 jumps,
\[
\nu(x,A)= \sum_{i=1}^d \lambda_i(x) \int 1_A ((y-e_i)x_i) \mu_i(dy),
\]
where $\lambda_i(x) = \frac{q_i(x)}{x_i}\mathds1_{\{ x_i\neq 0\}}$ for some nonnegative $q_i \in \Pol_1(E)$ such that $\lambda$ is not constant on $E\cap\{x_i\neq0\}$, and the measures $\mu_i$ are supported on $\Delta^d\setminus \{e_i\}$ and satisfy $\int |y|\mu_i(dy)<\infty$. Economically, this specification means that downward jumps occur with higher and higher intensity the closer the assets are to $0$, and can therefore be used to model downward spirals in stock prices. We require that for all $j\ne k$,
\[
\frac{\beta}{2} - \sum_{i\ne k}q_i(e_j)\int y_k\mu_i(dy) + q_k(e_j) \int \left( 1+ \log(y_k) - y_k\right) \mu_k(dy) > 0,
\]
which ensures that $X$ remains in the relative interior $\mathring{\Delta}^d$. This can be proved similarly as in \citet[Theorem~5.7]{Filipovic/Larsson:2016}. Furthermore, this model admits relative arbitrage opportunities. To see this, we argue that no equivalent probability measure can turn $X$ into a martingale. Indeed, Lemma~5.6 in \citet{C:16} implies that, under any martingale measure, $X$ must reach the relative boundary of $\Delta^d$ with positive probability on any time horizon, contradicting equivalence. Since no equivalent martingale measure exists for the market weights, the model admits relative arbitrage.

Clearly any other polynomial diffusion model on the simplex can be enhanced by jumps of this form, which yields a large class of tractable jump-diffusion models applicable in the realm of SPT.

\subsection{Valuation of defaultable zero--coupon bonds} \label{S:Def_appl}

Polynomial jump-diffusions on the unit interval can be brought to bear on default risk modeling. We consider the stochastic recovery rate framework of \cite{Jarrow1998} and \cite{KT:17}. For further references, see also \cite{fi:13}, \cite{zh:06}, and \citet[Chapter~7]{jb:09} for an overview, as well as \cite{du:04} for the classical approach using affine processes. Note also that polynomial diffusion models for credit risk have appeared in \cite{Ackerer/Filipovic:2016}.

Let $(\Omega,\Fcal,\F,\Q)$ with $\F=(\Fcal_t)_{t\geq0}$ be a filtered probability space satisfying the usual conditions. Here $\Q$ is a risk neutral  measure. Let $B = (B_t)_{t\geq0}$ be the value of the risk-free bank account with initial value of one monetary unit. For any $t\le T$, we denote by $\widetilde P(t,T)$
the price at time $t$ of a defaultable zero--coupon bond with maturity $T \geq 0$ and unit notional. Due to default risk, its actual payoff $\widetilde P (T,T)$ at maturity is random and lies between zero and one. Under the premise that all discounted defaultable zero--coupon bond prices $\widetilde P(t,T)/B(t)$ are true martingales under $\Q$, we get
$$\widetilde P(t,T)=\E_\Q\Big[\frac{B_t}{B_T} S_T\Big| \Fcal_t\Big],$$
where $S_t:=\widetilde P(t,t)$ is known as the {\em recovery rate}. Suppose now for simplicity that $B$ and $S$ are conditionally independent under $\Q$. Then
$$\widetilde P(t,T)=\E_\Q\bigg[\frac{B_t}{B_T}\bigg| \Fcal_t\bigg]\E_\Q[ S_T| \Fcal_t]=P(t,T)\E_\Q[ S_T| \Fcal_t],$$
where $P(t,T)$ is the price of a non-defaultable zero--coupon bond with maturity $T \geq 0$ and unit notional.

Motivated by the typically long and complicated unwinding process after a default occurs, \citet{KT:17} drop the assumption that the recovery rate $S$ is known when default happens. Excursions of $S$ below 1 are interpreted as liquidity squeezes resulting in a delay of due payments, which may or may not turn into a default. In this framework, the risk-neutral recovery rate $S$ typically starts with a constant trajectory at level 1.
Once the recovery has jumped below 1, it pursues an unsteady course. Downward moves of the recovery rate are self-exciting, as deterioration of the counterparty's credit quality typically makes full recovery more unlikely. Nonetheless, $S$ may return to 1 and remain there for some period of time.

A polynomial model for the recovery rate $S$ can be constructed as follows. Let $X$ be a polynomial jump-diffusion of Type 2 with ``no-jump'' point $x^*=0$. Assume that  $\kappa(1-\theta)=(1+q) \int y \mu(\de y)$; this condition guarantees that if $X$ reaches level 1, it can leave it only by means of a jump. More precisely, $X$ persists at level 1 until its first jump, which occurs according to an $(1+q)$--exponentially distributed stopping time and a downward $\mu$-distributed jump size.
Moreover, since the jump intensity is the positive branch of a hyperbola with a pole in zero, downward jumps of $X$ get more and more likely as the process approaches zero.

In view of the discussion above, polynomial transformations $S:=p(X)$ of $X$, where $p\in\Pol([0,1])$ is increasing and satisfies $p([0,1])\subseteq[0,1]$, are well-suited to describe the recovery rate. The polynomial property of $X$ permits to express the forward recovery rate $F(t,T)=\E_\Q[S_T\mid\Fcal_t]$ in closed form. We provide two concrete specifications, by choosing $p(x)=x$ and $p(x)=x^2$. In the first case, $S=X$, the moment formula \eqref{eq:momentformula} yields
$${F(t,T)}=\big(1-e^{-(T-t)\kappa}\big)\theta +e^{-(T-t)\kappa}S_t.$$
In the second case, $S=X^2$, we find
$${F(t,T)}=\frac{\big(\kappa(1-e^{(T-t)G_2})+G_2(1-e^{-(T-t)\kappa})\big)\theta}{\kappa+G_2}
 +\frac{G_1\big(e^{(T-t)G_2}-e^{-(T-t)\kappa}\big)}{\kappa+G_2}\sqrt{S_t}+
e^{-(T-t)\kappa}S_t,$$
where $G_1:=A+2\kappa\theta+\int y^2\mu(\de y)$ and $G_2:= -A-2\kappa+q\int y^2\mu(\de y)$.

\begin{appendices}

\section{Proof of Theorem~\ref{thm3}} \label{app11}

We assume that $\Gcal:\Pol(E)\to C(E)$ is a linear operator that satisfies the positive maximum principle and $\Gcal 1=0$.

Fix $x\in E$ and define the linear functionals $\Wcal_{ij}: \Pol(E-x) \to\R$ for $i,j=1,\ldots d$ by
\[
\Wcal_{ij}(p) := \Gcal\left( p(\fdot - x) e_i^\top (\fdot - x) e_j^\top (\fdot - x) \right)(x),
\]
as well as $\Wcal_u:\Pol(E-x)\to \R$ for $u\in\R^d$ by
\[
\Wcal_u(p) := \sum_{i,j=1}^d u_i u_j \Wcal_{ij}(p) = \Gcal\left( p(\fdot - x) (u^\top (\fdot - x))^2 \right)(x).
\]
Here and throughout the proof we view $u^\top (\fdot - x)$ as a polynomial on $E$ to avoid the more cumbersome notation $u^\top (\fdot - x)|_E$.

If $p\ge0$ on $E-x$, then $p(\fdot - x) (u^\top (\fdot - x))^2 \in \Pol(E)$ is minimal at $x$, which by the positive maximum principle yields $\Wcal_u(p)\ge0$. The Riesz-Haviland theorem, Lemma~\ref{lem5}, thus provides measures $\nu_u(x,\de\xi)$ concentrated on $E-x$ such that
\[
\Wcal_u(p) = \int p(\xi) \nu_u(x,\de\xi).
\]
By polarisation we have $\Wcal_{ij}=\frac{1}{2}(\Wcal_{e_i+e_j}-\Wcal_{e_i}-\Wcal_{e_j})$, whence
\[
\Wcal_{ij}(p) = \int p(\xi) \nu_{ij}(x,\de\xi),  \qquad \nu_{ij}=\frac{1}{2}(\nu_{e_i+e_j}-\nu_{e_i}-\nu_{e_j}).
\]
The triplet $(a,b,\nu)$ is now defined at $x$ by
\begin{equation} \label{eq:biaij}
a_{ij}(x) := \nu_{ij}(x,\{0\}), \qquad b_i(x) := \Gcal( e_i^\top (\fdot - x))(x),
\end{equation}
and
\[
\nu(x,\de\xi) := \frac{1}{|\xi|^2}\mathds1_{\{\xi\ne0\}}\left( \nu_{e_1}(x,\de\xi) + \cdots + \nu_{e_d}(x,\de\xi)\right).
\]
Next, observe that
\begin{align*}
\int \xi_i\xi_j p(\xi) |\xi|^2 \nu(x,\de\xi) &= \sum_{k=1}^d \int \xi_i\xi_j p(\xi) \nu_{e_k}(x,\de\xi) \\
&= \sum_{k=1}^d \Gcal\left( p(\fdot - x) e_i^\top (\fdot - x) e_j^\top (\fdot - x) (e_k^\top (\fdot - x))^2 \right)(x) \\
&= \Gcal\left( p(\fdot - x) e_i^\top (\fdot - x) e_j^\top (\fdot - x) |\fdot - x|^2 \right)(x) \\
&= \Wcal_{ij}( p\, |\fdot|^2 ) \\
&= \int p(\xi)|\xi|^2 \nu_{ij}(x,\de\xi),
\end{align*}
for all $p\in\Pol(E-x)$. By Weierstrass's theorem and dominated convergence, this actually holds for all $p\in C(E-x)$, whence $\mathds1_{\{\xi\ne0\}}\nu_{ij}(x,\de\xi)=\xi_i\xi_j\nu(x,\de\xi)$. Consequently,
\begin{equation} \label{eq:Wij(p)}
\Wcal_{ij}(p) = \int p(\xi) \xi_i \xi_j \nu(x,\de\xi) + p(0)a_{ij}(x).
\end{equation}
Consider now any polynomial $p\in\Pol(E-x)$, and choose a representative $q\in\Pol(\R^d)$, $p=q|_{E-x}$. Note that $q$ is of the form
\[
q(\xi) = c_0 + \sum_{i=1}^d c_i \xi_i + \sum_{i,j=1}^d \xi_i \xi_j q_{ij}(\xi)
\]
for some polynomials $q_{ij}\in\Pol(\R^d)$. Let $p_{ij}:=q_{ij}|_{E-x} \in \Pol(E-x)$. Then the linearity of $\Gcal$, the fact that $\Gcal 1=0$, and~\eqref{eq:biaij} and~\eqref{eq:Wij(p)} yield
\begin{align*}
\Gcal( p(\fdot - x))(x) &= c_0\, \Gcal 1(x) + \sum_{i=1}^d c_i\, \Gcal( e_i^\top (\fdot - x) )(x) \\
&\qquad + \sum_{i,j=1}^d \Gcal\left( p_{ij}(\fdot-x) e_i^\top (\fdot-x)e_j^\top (\fdot-x)\right)(x) \\
&= \sum_{i=1}^d c_i b_i(x) + \sum_{i,j=1}^d \left( \int q_{ij}(\xi) \xi_i \xi_j \nu(x,\de\xi) + q_{ij}(0) a_{ij}(x) \right) \\
&= \frac{1}{2}\tr\left( a(x) \nabla^2 q(0) \right) + b(x)^\top \nabla q(0) \\
&\qquad + \int \left( q(\xi) - q(0) - \xi^\top \nabla q(0) \right) \nu(x,\de\xi).
\end{align*}
Thus, with $p(\xi)=f(x+\xi)$ for a polynomial $f\in\Pol(E)$, we obtain the desired form~\eqref{eq:G LK}, where the right-hand side is computed using a representative of $f$, the choice of which is arbitrary.

It remains to verify that the $a$, $b$, and $\nu$ satisfy the additional stated properties. First, $a(x)$ is positive semidefinite since $u^\top a(x) u=\sum_{i,j=1}^d u_iu_j\nu_{ij}(x,\{0\})=\nu_u(x,\{0\})\ge0$, and $\nu$ clearly satisfies the support conditions $\nu(x,\{0\})=0$ and $\nu(x,(E-x)^c)=0$. Next, since $\Gcal$ maps polynomials to continuous functions, it is clear from~\eqref{eq:biaij} that $b$ is bounded and measurable. Similarly, $x\mapsto\int p(\xi) \nu_u(x,\de\xi)=\Gcal( p(\fdot - x) (u^\top (\fdot - x))^2 )(x)$ is continuous, hence bounded and measurable, for every $p\in\Pol(E)$, and so by the monotone class theorem $\nu_u(\fdot, A)$ is measurable for every Borel set $A\subseteq E-x$. Thus $\nu_u(x,\de\xi)$ is a kernel, from which it follows that $a$ is measurable and $\nu(x,\de\xi)$ is a kernel. Finally, continuity in $x$ of
\[
\tr( a(x) ) + \int |\xi|^2 \nu(x,\de\xi) = \nu_{e_1}(x,E-x) + \cdots + \nu_{e_d}(x,E-x) = \Gcal( |\fdot - x|^2)(x)
\]
implies that $a$ and $\int|\xi|^2\nu(\fdot,\de\xi)$ are bounded on $E$. \qed

\section{Proof of Theorem \ref{lem3} and Lemma \ref{lem7}}\label{appa2}

\begin{proof}[Proof of Theorem~\ref{lem3}]
Let $\alpha,\beta\in\R_+$ and  $\Gcal_1,\Gcal_2\in\Kcal$, and fix $f\in\Pol_k(E)$ for some $k\in\N$. Then, since $\Gcal_1 f, \Gcal_2 f\in\Pol_k(E)$, we have that
$$\Gcal f:=\alpha\Gcal_1 f+\beta \Gcal_2f\in\Pol_k(E)$$
as well, proving that it is polynomial. By Theorem~\ref{lem2}, the well-posedness of the martingale problem for $\Gcal$ follows directly by the well-posedness of the martingale problems of $\Gcal_1$ and $\Gcal_2$.
For the second part, set $\left(\Gcal_n\right)_{n\in\N}$ as in the statement of the theorem and recall that $\Pol_k(E)$ is closed under pointwise 
convergence for each $k\in\N_0$.  Fixing $f\in\Pol_k(E)$, since $\Gcal_n f\in\Pol_k(E)$  by the polynomial  property of $\Gcal_n$, 
we can conclude that $\Gcal f\in\Pol_k(E)$ as well. Again, existence and uniqueness of a solution to the martingale problem is guaranteed by Theorem~\ref{lem2}, since
$\mathcal{G}1=0$ and the positive maximum principle is preserved in the limit.
\end{proof}

\begin{proof}[Proof of Lemma~\ref{lem7}]
In order to prove the first part of the lemma, it is enough to observe that for all $n\in\N$ and $|\bk|\geq1$,
\begin{equation}\label{eqn59}
\Gcal_n\big((\fdot-x)^\bk\big)(x)=\begin{cases}
b^n_i(x)&\text{if } \bk=e_i,\\
a^n_{ij}(x)+\int\xi_i\xi_j\nu_n(x,\de\xi)&\text{if } \bk=e_i+e_j,\\
\int\xi^ \bk\nu_n(x,\de\xi)&\text{if } |\bk|\geq3.
\end{cases}
\end{equation}
For the second part of the lemma, if $\Gcal$ is well-defined we know from Theorem~\ref{lem3} that it has a \lkr, for some coefficients $a$ and $b$, and a kernel $\nu(x,\de\xi)$. As a result, the analog of (\ref{eqn59}) holds true for $\Gcal$ and by definition of the limit we thus obtain
$$b_i(x)=\Gcal\big((\fdot-x)^\bk\big)(x)=\lim_{n\to\infty}\Gcal_n\big((\fdot-x)^\bk\big)(x)=\lim_{n\to\infty}b_i^n(x),$$
and similarly
\begin{equation}\label{eqn67}
a_{ij}(x)\mathds1_{\{\bk=e_i+e_j\}}+\int\xi^\bk\nu(x,\de\xi)=\lim_{n\to\infty}\bigg(a_{ij}^n(x)\mathds1_{\{\bk=e_i+e_j\}}+\int\xi^\bk\nu_n(x,\de\xi)\bigg),
\end{equation}
for all $|\bk|\geq2$. Since $\nu(x,\de\xi)$ does not have mass in 0, $\nu(x,\de\xi)=|\xi|^{-4}\big(|\xi|^4\nu(x,\de\xi)\big)$. Moreover, using that 
moments completely determine compactly supported finite distribution, the kernel $|\xi|^4\nu(x,\de\xi)$, and thus $\nu(x,\de\xi)$, is uniquely determined by (\ref{eqn67}). \end{proof}

\section{Proof of Theorem~\ref{T:pol is affine}}\label{secc}

Throughout the proof we assume without loss of generality that $E$ has nonempty interior.
Suppose that $\nu(x,\de\xi)$ is the zero measure for all $x\in E$. Setting $\lambda=0$, 
the form of $\gamma(x,y)$ and the measure $\mu$ are irrelevant and we are thus free to choose $K\leq1$.
We may therefore suppose that $\nu(x,\fdot)$ is nonzero for some $x\in E$, and thus in particular 
 \begin{equation}\label{eqnn4}
 \mu\big(\gamma(x,\fdot)\neq0\big)>0
 \end{equation}
 for at least one $x\in E$. As in Remark \ref{rem3}, we can assume without loss of generality that $\mu$ is compactly supported 
 and hence all its moments of order at least two are finite. Set then
 \begin{align*}
p_\bn&:=\int\gamma^\bn(\fdot,y)\mu(\de y)\qquad\text{and}\qquad
r_\bn:=\int\xi^\bn\nu(\fdot,\de\xi)=\lambda p_\bn
\end{align*}
and note that, by the integrability conditions on $\mu$ and condition \eqref{eq:pol is affine:2} respectively,  $p_\bn$ and $r_\bn$ are polynomials on $E$ for all $|\bn|\geq3$. In particular,
 $p_{4e_i}$ is a nonzero polynomial for at least one $i\in\{1,\ldots,d\}$ by (\ref{eqnn4}), and thus
 \begin{equation}\label{eqn60}
 \lambda(x)=\frac{r_{4e_i}(x)}{p_{4e_i}(x)}
 \end{equation}
 for all $x\in E\setminus\{p_{4e_i}=0\}$.
 
Since $E$ has nonempty interior by assumption, each polynomial $p\in\Pol(E)$ has a unique representative $\overline p\in\Pol(\R^d)$ such that $\overline p|_E=p$. In particular, the degree of a polynomial on $E$ always coincides with the maximal degree of its monomials. Assume now for contradiction that $K$ cannot be chosen less than or equal one. Let
$$\bn_j:=\sup\{\bk: \mu(y_{\bk}^j\neq0)\neq0\}$$
 be the multi-index corresponding to the leading monomial of $\gamma_j(x,y)$, with respect to some graded lexicographic order. 
 Choose $j\in\{1,\ldots,d\}$ such that $|\bn_j|\geq2$ and note that by the maximality of $\bn_j$ and since $\int (y_{\bn_j}^j)^{{\kk}}\mu(\de y)>0$ we have that
$$\deg(p_{{\kk}e_j})=\deg\bigg(\int (y_{\bn_j}^j)^{{\kk}}\mu(\de y)x^{{\kk}\bn_j}\bigg)={\kk}|\bn_j|.$$
Analogously, $\deg(p_{4e_j})=4|\bn_j|$ and thus (\ref{eqn60}) holds true for $i=j$. 
Since ${p_{{\kk}e_j}(x)}r_{4e_j}(x)=p_{4e_j}(x)r_{{\kk}e_j}(x)$, using that $|\bn_j|\geq2$ we can compute
$$\deg(p_{10e_j}r_{4e_j})\geq10|\bn_j|>4|\bn_{j}|+10\geq\deg(p_{4e_j}r_{{\kk}e_j}),$$
and obtain the desired contradiction. As a result, $K$ can always be chosen smaller than or equal one.
\qed


\section{The unit interval: Proof of Lemma~\ref{thm4} and Theorem~\ref{thm1}}\label{app21}
\begin{proof}[Proof of Lemma~\ref{thm4}]

Assume $\Gcal$ is a polynomial  operator and its martingale problem is well-posed. Theorem~\ref{lem2} and Theorem~\ref{thm3} imply that $\Gcal$ is of L\'evy type for some triplet $(a,b,\nu)$, so that in particular (i) holds. Condition (iii) follows from Lemma~\ref{lem1}. To verify (ii), let $f_n$ be polynomials on $[0,1]$ with $0\le f_n\le 1$, $f_n(0)=1$, $xnf_n(x)\le 1$, and $f_n(x)\downarrow 0$ for $x\in(0,1]$. For example, one can choose $f_n(x):= \frac {n-1}n(1-x)^n+\frac 1 n$. Let $g_n(x):=\frac x n-x^2f_n(x)$. Then $g_n$ has a minimum at $x=0$, so by the positive maximum principle,
\[
0\le \Gcal g_n(0) = -a(0) + \frac{1}{n}b(0) - \int f_n(\xi)\xi^2 \nu(0,\de\xi) \to -a(0), \qquad n\to\infty,
\]
where the dominated convergence theorem was used to pass to the limit.
Similarly, $h_n(x):=xf_n(x)$ is nonnegative on $[0,1]$ with a minimum at $x=0$, so the monotone convergence theorem yields
\[
0 \le \Gcal h_n(0) = b(0) - \int \xi (1-f_n(\xi)) \nu(0,\de\xi) \to b(0) - \int \xi \nu(0,\de\xi), \qquad n\to\infty.
\]
 We have thus shown (ii) for the boundary point $x=0$. The case $x=1$ is similar.


We now prove the converse. Lemma~\ref{lem1} and (iii) imply that $\Gcal$ is polynomial. Next, clearly $\Gcal 1=0$. Thus, by Theorem~\ref{lem2} it only remains to verify the positive maximum principle in order to deduce that the martingale problem for $\Gcal$ is well-posed. To this end, let $f\in \Pol(E)$ be an arbitrary polynomial having a maximum over $E$ on some $x\in E$.
If $x\in\Int(E)$ it follows that $f'(x)=0$, $f''(x)\leq0$, and $f(x)\geq f\big(x+\xi\big)$ for all $\xi\in E-x$. Hence, using that $a\geq0$ on $E$, we conclude that $\Gcal f(x)\leq0$. 
On the other hand, if $x\in\partial E=\{0,1\}$ we use that $a(x)=0$ and the integrability of $\xi$ with respect to $\nu(x,\fdot)$ to write
$$\Gcal f(x)= \left(b(x)-\int\xi \nu(x,\de \xi)\right) f'(x)+\int \left( f(x+\xi)-f(x) \right) \nu(x,\de \xi).$$
The Karush-Kuhn-Tucker conditions (see e.g.~Proposition~3.3.1 in \citet{B:99}) imply that $f'(x)\leq0$ if $x=0$ and $f'(x)\geq0$ if $x=1$, and thus the first summand is nonnegative by (ii). Using as before that $f(x)\geq f\big(x+\xi\big)$ for all $\xi\in E-x$, we conclude that $\Gcal f(x)\leq0$.
\end{proof}

\begin{proof}[Proof of Theorem~\ref{thm1}]
By assumption $\Gcal$ is polynomial and its martingale problem is well-posed. Hence, conditions (i)-(iii)  of Lemma~\ref{thm4} are satisfied.
As in Remark \ref{rem3}, we can assume without loss of generality that $\mu$ is compactly supported. In particular, all its moments of order at least two are finite. For all $n\geq2$ set then
\begin{equation}\label{eqn7}
p_{n}:=\int\gamma^n(\fdot,y)\mu(\de y)\qquad\textup{and}\qquad r_n:=\int\xi^n\nu(\fdot,\de\xi)=\lambda p_n.
\end{equation}
Note that $p_n\in\Pol_n(E)$ for all $n\geq2$ by the integrability conditions on $\mu$, and $r_n\in\Pol_n(E)$ for all $n\geq3$ by condition (iii) of Lemma~\ref{thm4}.
By Remark \ref{rem10} we know that
$$\lambda(x)=\frac {r_4(x)}{p_4(x)}\mathds1_{\{p_4(x)\neq0\}}$$
and hence the condition $\nu\big(x,(E-x)^c\big)=0$ of \eqref{thenu} implies that $\mu$ can be chosen to be supported on $[0,1]^2$ and such that $\gamma(x,y)=y_1(-x)+y_2(1-x)$ $\mu$-a.s. 
By Lemma~\ref{thm4}(iii) we also know that $b\in\Pol_1(E)$ and, by Lemma~\ref{thm4}(ii), that the boundary conditions
\begin{equation}\label{eqn69}
b(0)\geq\lambda(0)\int \gamma(0,y)\ \mu(\de y)\qquad\textup{and}\qquad b(1)\leq\lambda(1)\int \gamma(1,y)\ \mu(\de y)
\end{equation}
hold. 
We consider now five complementary assumptions, which will lead to Types~0 to~4.

Assume that $\nu(x,\de\xi)=0$. Then Lemma~\ref{thm4} implies that $a(x)=Ax(1-x)$ for some $A\in\R_+$. This proves that $\Gcal$ is an operator of Type~0.

Assume now that $\nu(x,\de\xi)\neq0$ and $\lambda$ can be chosen to be constant. We can then without loss of generality set $\lambda=1$. 
Moreover, since in this case $r_{2}\in\Pol_2(E)$, we can conclude as before that $a(x)=Ax(1-x)$ for some $A\in\R_+$. This proves that $\Gcal$ is an operator of Type~1.

Assume that $\nu(x,\de\xi)\neq0$, $\lambda$ cannot be chosen to be constant, and $p_4(x^*)=0$ for some $x^*\in\R$. 
By definition of $p_4$ this automa\-tically implies that $\gamma(x^*,y)=0$, and in particular $x^*=y_2(y_1+y_2)^{-1}$, for $\mu$-a.e.~$y\in[0,1]^2$. As a result, $x^*$ lies in $E$, and setting $\overline y:=y_1+y_2$ we obtain $\gamma(x,y)=-\overline y(x-x^*)$. Moreover, since $\overline y=\frac{y_2}{x^*}=\frac{y_1}{1-x^*}$ $\mu$-a.s., 
we can conclude that it is square-integrable and takes values in the set $[0,(x^*\vee(1-x^*))^{-1}]$ $\mu$-a.s.
By (\ref{eqn7}) we can then write
$$\lambda(x)=\frac{r_3(x)}{p_{3}(x)}=\frac{\overline r_3(x)}{(x-x^*)^3}$$
for some $\overline r_3\in\Pol_3(E)$, for all $x\in E\setminus\{x^*\}$. Since in this case $\nu(x^*,\fdot)=0$ we are free to choose $\lambda(x^*)=0$.
By Lemma~\ref{thm4} we also know that $r_2$ is bounded on $E$. Therefore, noting that
$$r_2(x)=\frac{\overline r_3(x)}{(x-x^*)}\int(-\overline y)^2\mu(\de \overline y) \quad\text{for all $x\in E\setminus\{x^*\}$,}$$
it follows that $\overline r_3(x^*)=0$, and thus $\lambda(x)=\frac{q_2(x)}{(x-x^*)^2}$ for some $q_2\in\Pol_2(E)$ and all in $x\in E\setminus\{x^*\}$. This in particular implies that $r_2\in\Pol_2(E\setminus\{x^*\})$ and hence $a\in\Pol_2(E\setminus\{x^*\})$. Knowing that $a+r_2$ has to be continuous by condition (iii) of Lemma~\ref{thm4}, we can finally  deduce that
$$a(x^*)=\lim_{x\to x^*}a(x)+{q_2(x^*)}\int\overline y^2\ \mu(\de \overline y).$$

Suppose now $x^*\in\{0,1\}$. Then,  since $a\geq0$ on $E$ and $a(0)=a(1)=0$, we conclude that $q_2(x^*)=0$, $a(x)=Ax(1-x)$ for some $A\in\R_+$, and thus $\lambda(x)=\frac {q_1(x)}{(x-x^*)}\mathds1_{\{x\neq x^*\}}$ for some $q_1\in\Pol_1(E)$. For $x^*=0$, respectively $x^*=1$, the nonnegativity of $\lambda$ implies that $q_1(x)=1+qx$, respectively $q_1(x)=1+q(1-x)$, for some $q\in[-1,\infty)$.  As a result, $\Gcal$ is an operator of Type~2.

On the other hand, if $x^*\in(0,1)$, using again that $a\geq0$ on $E$ and $a(0)=a(1)=0$, we conclude that
$$a(x)=Ax(1-x)+\left({q_2(x^*)}\int\overline y^2\ \mu(\de \overline y)\right)\mathds1_{\{x=x^*\}},$$
for some $A\in\R_+$, proving that $\Gcal$ is an operator of Type~3.

Assume now that $\nu(x,\de\xi)\neq0$,  $\lambda$ cannot be chosen to be constant, and $p_4(x)\neq0$ for all $x\in\R$. We must argue that $\Gcal$ is then necessarily of Type~4.
By \eqref{eqn7} we have $\lambda(x)=\frac{r_4(x)}{p_4(x)}$ on $E$, and thus on $\R$. Consequently, $\lambda$ is locally bounded, nonnegative, and non-constant. Moreover, \eqref{eqn7} yields the expression $\lambda(x)=\frac{r_3(x)}{p_{3}(x)}$  for all $x\in E$ with $p_{3}(x)\neq0$. These facts combined with the fundamental theorem of algebra imply that
\begin{equation} \label{eqn_ML1}
\lambda(x)=\frac{q_2(x)}{(x-\alpha)(x-\overline \alpha)}
\end{equation}
for some positive $q_2\in\Pol_2(E)$ and $\alpha\in\C\setminus\R$. Without losing generality we choose to satisfy $\im(\alpha)<0$. Note that no further cancellation of polynomial factors is possible in~\eqref{eqn_ML1} since $\lambda$ is non-constant.
Furthermore, condition (iii) of Lemma~\ref{thm4} yields $r_n\in\Pol_n(E)$ for all $n\geq 3$. Therefore, since $p_n(x)q_2(x)=(x-\alpha)(x-\overline \alpha)r_n(x)$ due to~\eqref{eqn7} and~\eqref{eqn_ML1}, it follows that $p_n(x)=(x-\alpha)(x-\overline\alpha)R_{n-2}(x)$ for all $x\in E$ and $n\geq3$, for some $R_{n-2}\in\Pol_{n-2}(E)$. This already implies \eqref{case4} for $n \geq 3$, i.e.,
\begin{equation} \label{eqn_ML2}
 p_n(\alpha)=\int \gamma^n(\alpha,y) \mu(dy)= \int (y_1 (-\alpha) +y_2(1-\alpha))^n \mu(dy)=0, \quad n \geq 3.
\end{equation}

Next, we will establish~\eqref{case4} also for $n=2$. In preparation for an application of Lemma~\ref{lem9} later on, choose a constant $C_\alpha$ such that 
\begin{align}\label{eq:Calpha}
\bigg|1+\frac{i\gamma(\alpha,y)}{C_\alpha}\bigg|<1\quad\text{and}\quad \frac{2|\gamma(
\alpha,y)|}{C_\alpha}<\tan^{-1}\bigg|\frac{\text{Im}(1-\alpha)}{\text{Re}(1-\alpha)}\bigg|
\end{align}
for all $y\in[0,1]^2\setminus\{0\}$.
Define
\[
f_k(z):=1-\Big(1+\frac{iz}{C_\alpha}\Big)^k,
\]
and note that $|f_k(\gamma(\alpha,y))|\leq2$ and $\lim_{k\to\infty}f_k(\gamma(\alpha,y))=1$ for all $y\in[0,1]^2\setminus\{0\}$. 
By dominated convergence we then obtain
\begin{equation}\label{eqn53}
\int \gamma^2(\alpha,y) \mu(\de y)=\lim_{k\to\infty}\int \gamma^2(\alpha,y)f_k\big(\gamma(\alpha,y)\big) \mu(\de y) = 0,
\end{equation}
where the last equality follows since, for each $k\ge3$, the integral on the right-hand side is a linear combination of 
$\int \gamma^n(\alpha,y) \mu(dy)$ with $3 \leq n \leq k$, and therefore vanishes due to~\eqref{eqn_ML2}. Hence, \eqref{case4} holds also for $n=2$.

We now derive some consequences. First, $r_2\in\Pol_2(E)$ and hence $a(x)=Ax(1-x)$ for some $A\in \R_+$.
Moreover, since
$$\Arg\big(\gamma^2(\alpha,\fdot)\big)\subseteq\big[2\Arg(1-\alpha),2\Arg(-\alpha)\big]$$
holds $\mu$-a.s., (\ref{eqn53}) implies that $\Arg(-\alpha)-\Arg(1-\alpha)\geq\pi/2$, or equivalently, $|2\alpha-1|\leq1$. In the description of Type~4 it is claimed in addition that the inequality is strict, i.e.
\begin{equation} \label{eqn0000001}
|2\alpha-1|<1.
\end{equation}
To see this, observe that in case of equality, (\ref{eqn53}) would imply that $\Arg(\gamma(\alpha,y))\subseteq \{\Arg(-\alpha),\Arg(1-\alpha)\}$ for $\mu$-a.e. $y\in[0,1]^2$, which is clearly incompatible with having $\int\gamma^3(\alpha,y)\mu(\de y)=0$ for some nontrivial measure~$\mu$.

Next, we claim that $\int y_1 \mu(\de y)=\int y_2 \mu(\de y)=\infty$. We prove this by excluding the complementary possibilities. First, assume for contradiction that $\int y_1 \mu(\de y)<\infty$ and $\int y_2 \mu(\de y)<\infty$. Then $\int |\gamma(\alpha,y)|\mu(\de y)<\infty$. Proceeding as with~\eqref{eqn53} we then deduce
$$\int \gamma(\alpha,y) \mu(\de y)=\lim_{k\to\infty}\int \gamma(\alpha,y)f_k\big(\gamma(\alpha,y)\big) \mu(\de y)=0,$$
which is clearly not possible since $\text{Im}(\gamma(\alpha,y))>0$ $\mu$-a.s. This is the desired contradiction.

Suppose instead $\int y_1 \mu(\de y)<\infty$ and $\int y_2 \mu(\de y)=\infty$. Define
\begin{equation}\label{eqn65}
g_k(y):=\textup{Re}\big(\gamma(\alpha,y)f_k\big(\gamma(\alpha,y)\big)\big).
\end{equation}
Set $C:={\text{Im}(1-\alpha)}({\text{Re}(1-\alpha)})^{-1}$ and observe that $C>0$ due to the fact that $\text{Re}(1-\alpha)>0$ in view of \eqref{eqn0000001}. Then define the set
$$A:=\left\{y\in[0,1]^2:\frac{\textup{Im}\big(\gamma(\alpha,y)/C_\alpha\big)}{\textup{Re}\big(\gamma(\alpha,y)/C_\alpha\big)}\in 
[C,2C],\ \frac{2|\gamma(\alpha,y)|}{C_\alpha}\leq{\tan^{-1}(C)},\ \textup{Im}\Big(\frac{\gamma(\alpha,y)}{C_\alpha}\Big)\geq0 \right\}.$$
Choosing $\e>0$ small enough such that $\{y_1<\e y_2\}\cap[0,1]^2 \subseteq A$, by Lemma \ref{lem9} we have that
 $$\{y_1<\e y_2\}\cap[0,1]^2\subseteq A\subseteq \left\{h_k\bigg(\frac{\gamma(\alpha,y)}{C_\alpha}\bigg)\geq0\right\}=\left\{\frac{g_k(y)}{C_\alpha}\geq0\right\}= \{g_k(y)\geq0\}.$$
 We can then compute
 $$g_k(y)\geq-2|\gamma(\alpha,y)|\mathds1_{\{g_k(y)<0\}} \geq
-2|\gamma(\alpha,y)|\mathds1_{\{y_1\geq \e y_2\}} \geq
-2(y_1|\alpha|+\e^{-1}y_1|1-\alpha|),$$
for all $k\in \N$ and $\mu$-a.e.~$y\in[0,1]^2$. Fatou's lemma then yields
$$0=\lim_{k\to\infty}\int g_k(y)\mu(\de y)\geq -\textup{Re}(\alpha)\int y_1\mu(\de y)+\textup{Re}(1-\alpha)\int y_2\mu(\de y)=\infty,$$
using in the last step that $\text{Re}(1-\alpha)>0$. Again we arrive at a contradiction.

Finally, suppose $\int y_1 \mu(\de y)=\infty$ and $\int y_2 \mu(\de y)<\infty$. We may then repeat the arguments from the first case, using the function $-g_k$ instead of $g_k$ to obtain the required contradiction. In summary, we have shown that $\int y_1 \mu(\de y)=\int y_2 \mu(\de y)=\infty$, as claimed.

Finally, the boundary conditions~\eqref{eqn69} now forces $\lambda(0)=\lambda(1)=0$, which in view of~\eqref{eqn_ML1} yields $q_2(0)=q_2(1)=0$. Therefore $q_2(x)=Lx(1-x)$ for some constant $L>0$, as claimed. As a result, $\Gcal$ is an operator of Type~4, and the proof of Theorem~\ref{thm1} is complete.
\end{proof}

\begin{lemma}\label{lem9}
Fix $C>0$ and set 
$h_k(z):=\textup{Re}\big(z(1 -(1 +iz)^k)\big)$
for all $k\in\N$. Then there is some $K\in\N$ such that
$$\Big\{\textup{Im}(z)/\textup{Re}(z)\in [C, 2C],\ 2|z|\leq{\tan^{-1}(C)},\ \textup{Im}(z)\geq0\Big\}\subseteq\{h_k(z)\geq0\}$$
for all $k\geq K$.

\begin{proof}
Fix $c\in[C,2C]$ and let $z\in\C$ such that $\textup{Im}(z)=c\textup{Re}(z)$ and $\text{Im}(z)\geq0$. Define then $w:=1+iz$ and compute
$$
h_k(z)=h_k(i-iw)=\textup{Im}(w)\big(1-\textup{Re}(w^k)+c\textup{Im}(w^k)\big).
$$
Let $x:= \Arg(w)$, note that $x=\Arg(1+iz)\in \big[0,2|z|\big]$ and moreover
\begin{equation}\label{eqn68}
\textup{Re}(w^k)-c\textup{Im}(w^k)=\frac{\cos(kx)-c\sin(kx)}{\big(c\sin(x)+\cos(x)\big)^{k}}=\frac{\sqrt{1+c^2}\cos(kx+\tan^{-1}(c))}{\big(\sqrt{1+c^2}\cos(x-\tan^{-1}(c))\big)^k}.
\end{equation}
Since $\text{Im}(w)=\text{Re}(z)\geq 0$, it is then enough to show that for $k$ big enough this expression is smaller than or equal to 1 for all $x\in\big[0, \tan^{-1}(C)\big]$ and $c\in[C,2C]$.

Let $x_k^c:=(\pi-\tan^{-1}(c))/k$ be the first minimum of the numerator. Observe that for $x=x_k^c$ the denominator converges to $\exp\big(c(\pi-\tan^{-1}(c))\big)>\sqrt{1+c^2}$ uniformly on compact sets. As a result, for $k$ big enough, we have that
$$\sup_{c\in[C,2C]}\frac{\cos(kx)-c\sin(kx)}{\big(c\sin(x)+\cos(x)\big)^{k}}\leq \sup_{c\in[C,2C]}\frac{\cos(kx)-c\sin(kx)}{\big(c\sin(x^c_k)+\cos(x^c_k)\big)^{k}}\leq1,$$
for all $x\in[x_k^c,\tan^{-1}(C)]$.
Since expression \eqref{eqn68} takes value $1$ for $x=0$ and is decreasing in $x$ on $[0,x_k^c]$, we conclude the proof.
\end{proof}
\end{lemma}

\section{The unit simplex: Proof of Lemma~\ref{thm6} and Theorem \ref{th:mainsimplex}}\label{app22}

\begin{proof}[Proof of Lemma~\ref{thm6}]

Assume $\Gcal$ is a polynomial  operator and its martingale problem is well-posed. Theorem~\ref{lem2} and Theorem~\ref{thm3} imply that $\Gcal$ is of L\'evy type for some triplet $(a,b,\nu)$, so that in particular (i) holds. Condition (iv) follows from Lemma~\ref{lem1}.
To prove (ii), fix $x\in E\cap\{x_i=0\}$.
Let $g^i_n(x):=g_n(x_i)$ and $h^i_n(x):=h_n(x_i)$, where $g_n$ and $h_n$ are the functions on $[0,1]$ described in the proof of Lemma~\ref{thm4}. Then by the positive maximum principle we conclude that 
\[
 0\leq\Gcal g^i_n(x) \to-a_{ii}(x)\qquad\text{and}\qquad \qquad 0\leq \Gcal h^i_n(x) \to b_i(x) - \int \xi_i \nu(x,\de\xi).
\]
The positive semidefiniteness of $a(x)$ then implies that $a_{ij}(x)=0$ for all $j\in\{1,\ldots, d\}$.
In order to verify (iii), note that setting $f^\Delta(x):=x^\top\mathbf1-1$, by the positive maximum principle we have
\begin{align*}
0=\Gcal(f^\Delta)(x)=b(x)^\top\mathbf1\qquad\text{and}\qquad
0=\Gcal\big((\fdot-x)^{e_j}f^\Delta\big)(x)=a_j(x)^\top\mathbf1,
\end{align*}
where $a_j(x)$ denotes the $j$th column of $a(x)$.

Conversely, Lemma~\ref{lem1} and (iv) imply that $\Gcal$ is polynomial.
Thus by Theorem~\ref{lem2}, the martingale problem for $\Gcal$ is well-posed, provided that $\Gcal 1=0$ and  $\Gcal$ satisfies the positive maximum principle. The first condition is clearly satisfied. For the second one, let $g^i(x):=x_i$ and $f\in \Pol(E)$ be an arbitrary polynomial having a maximum over $E$ at some $x\in E$. Observe that 
$$E=\big\{x\in\R^d\ :\ f^\Delta=0\text{ and }g^i\geq0 \text{ for all }i\in\{1,\ldots,d\}\big\},$$
and let $I(x)$ be the set of all active inequality constraints at point $x$, that is, $I(x)$ is the set of all $i\in\{1,\ldots,d\}$ such that $x_i=0$.
By the necessity of the Karush-Kuhn-Tucker conditions (see e.g.~Proposition~3.3.1 in \citet{B:99}), there exist multipliers $\mu\in\R_+^d$ and $\lambda\in\R$ such that $\mu_i=0$ for all $i\in\{1,\ldots,d\}\setminus I(x)$,
$$\nabla f(x)=-\sum_{i\in I(x)}\mu_i\nabla g^i(x)+\lambda \nabla f^\Delta(x)=-\sum_{i\in I(x)}\mu_ie_i+\lambda \mathbf1,$$
 and $v^\top\nabla^2f(x)v\leq0$
for all $v\in\R^d$ such that $ v^\top \mathbf 1=0$ and $v_i=0$ for all $i\in I(x)$.
Since $\xi^\top\mathbf1=0$ for $\nu(x,\fdot)$-a.e.~$\xi$, $b(x)^\top\mathbf1=0$ by (iii), and $\int|\xi_i|\nu(x,\de\xi)<\infty$ for all $i\in I(x)$ by (ii), we can thus write
\begin{align*}
\Gcal f(x) &= \frac{1}{2}\tr\left( a(x) \nabla^2 f(x) \right) +\sum_{i\in I(x)}-\mu_i\left(b_i(x)-\int\xi_i\nu(x,\de\xi)\right)\\
&\qquad + \int \left( f(x + \xi) - f(x) \right) \nu(x,\de\xi).
\end{align*}
We must argue that $\Gcal f(x)\le0$. The second term on the right-hand side is nonpositive by (ii). The last term is also nonpositive since $f$ is maximal over $E$ at $x$. It remains to show that the first term is nonpositive. To this end, let $\sqrt{a(x)}$ denote the symmetric and positive semidefinite square root of $a(x)$. Condition (iii) yields $a(x)\mathbf 1=0$ and thus $\sqrt{a(x)}\mathbf 1=0$. By symmetry of $\sqrt{a(x)}$ we deduce
$$\big(\sqrt{a(x)}v\big)^\top\mathbf 1=v^\top \sqrt{a(x)}\mathbf 1=0 \quad\text{for all $v\in\R^d$.}$$
Moreover, by (ii) we also have that $a(x)e_i=0$, and hence $\sqrt{a(x)}e_i=0$,  for all $i\in I(x)$. This implies that $\big(\sqrt{a(x)}\big)_{ij}=0$ and thus $\big(\sqrt{a(x)}v\big)_i=0$,  for all $i\in I(x)$ and $v\in \R^d$.
As a result,
$$v^\top \big(\sqrt{a(x)}\nabla^2f(x)\sqrt{a(x)}\big)v=\big(\sqrt{a(x)}v\big)^\top\nabla^2f(x)\big(\sqrt{a(x)}v\big)\leq0,$$
which implies that $\sqrt{a(x)}\nabla^2f(x)\sqrt{a(x)}$ is negative semidefinite. This gives the desired inequality
$$\Tr\big(a(x)\nabla^2f(x)\big)=\Tr\big(\sqrt{a(x)}\nabla^2f(x)\sqrt{a(x)}\big)\leq0,$$
showing that $\mathcal G f (x)\leq0$. This completes the proof of the lemma.
\end{proof}

Before starting the proof of Theorem \ref{th:mainsimplex}, we prove three auxiliary lemmas.

\begin{lemma}\label{clem1}
Consider a polynomial $p\in\Pol_n(E)$.
\begin{itemize}
\item[(a)] If $p$ vanishes on $E\cap\{x_i=x_j=0\}$, it can be written as
\begin{equation}\label{eqn62}
p(x)=x_i p_{n-1}^i(x)+x_jp_{n-1}^j(x)\quad\text{for some}\quad p_{n-1}^i,p_{n-1}^j\in\Pol_{n-1}(E).
\end{equation}
\item[(b)] If $p$ vanishes on $E\cap\big(\{x_i=0\}\cup \{x_j=0\}\big)$ for some $i\neq j$, it can be written as
\begin{equation}\label{eqn63}
p(x)=x_ix_jp_{n-2}(x) \quad\text{for some}\quad p_{n-2}\in\Pol_{n-2}(E).
\end{equation}
\item[(c)] If $p$ vanishes on $E\cap\big\{cx_i=x_j\}$ for some $c\geq0$ and $i\neq j$, it can be written as
\begin{equation}\label{eqn64}
p(x)=(-cx_i+x_j)p_{n-1}(x) \quad\text{for some}\quad p_{n-1}\in\Pol_{n-1}(E).
\end{equation}
\end{itemize}
\end{lemma}
\begin{proof}
Since every affine function on $E$ can be written as a linear one, there is a real collection $(p_{\bn})_{|\bn|=n}$ such that 
$p(x)=\sum_{|\bn|=n}p_{\bn}x^\bn,$
for all $x\in E$. Observe that for all $x\in E\cap\{x_i= x_j=0\}$ we have that
$$0=p(x)=\sum_{|\bn|=n\atop n_i=n_j=0} p_{\bn} x^\bn.$$
Assume without loss of generality that $i=d$ and $j=d-1$ (resp. $i=j=d$ if $i=j$) and note that, the polynomial $q\in\Pol(\R^{j-1})$ given by
$$q(x):=\sum_{|\bn|=n\atop n_i=n_j=0} p_{\bn} x^{\overline \bn},$$
where $\overline \bn=\sum_{k=1}^{j-1}n_ke_k$, is a homogeneous polynomial vanishing on the simplex. This directly implies that $q=0$ and hence $p_{\bn}=0$ for all $|\bn|=n$ such that $n_i=n_j=0$. We can thus conclude that $p$ satisfies (\ref{eqn62}).
 
Proceeding as before for the second part,
we obtain that $p_\bn=0$  for all $|\bn|=n$ such that $n_i=0$ or $n_j=0$ and can thus conclude that $p$ satisfies (\ref{eqn63}).

Finally, for the third part it is enough to note that the polynomial $\overline p\in\Pol_n(E)$ given by
$$\overline p(x):=p \Big(x+\frac {cx_i}{1+c}(e_j-e_i)\Big)$$
vanishes on $E\cap\{x_j=0\}$. By (a) this gives us that
$$p(x)=\overline p\big(x+ {cx_i}(e_i-e_j)\big)=(-cx_i+x_j)\overline p^j_{n-1}(x)$$
on $E\cap\{x_j\geq cx_i\}$ and thus on $E$, proving that $p$ satisfies (\ref{eqn64}).
\end{proof}

\begin{lemma}\label{lem:formgamma}
Let $\mu$ be a nonzero measure on $\big(\Delta^d\big)^d\setminus\{e_1,\ldots,e_d\}$. The function $\gamma: E\times (\Delta^d)^d\to\R$ given by $\gamma(x,y)=\sum_{i=1}^d(y^i-e_i) x_i$ can be represented as 
\begin{align}\label{eq:linpolform}
\gamma(x,y)=H(y)P_1(x)\quad \mu\textrm{-a.s.},
\end{align}
for a measurable function $H:(\Delta^d )^d \to \mathbb{R}^d$ and $P_1 \in \Pol_1(E)$, if and only if one of the following cases holds true
\begin{enumerate}
\item[(a)] $\gamma(x,y)=(y^i-e_i)x_i$ $\mu\textrm{-a.s.}$ for some $i \in \{1, \ldots, d\}$.
\item[(b)] $\gamma(x,y)=y^j_i (-cx_i+x_j)(e_i-e_j)$ $\mu\textrm{-a.s.}$ for some $i\neq j$ and $c>0$. In this case $y^j_i\in(0,\frac 1 c\land1]$ $\mu$-a.s. 
\end{enumerate} 
\end{lemma}

\begin{proof}
First assume that (\ref{eq:linpolform}) holds true.
Since $P_1 \in \Pol_1(E)$, and as every affine function on $E$ has a linear representation, we can write
$
P_1(x)=C^{\top}x, 
$
for some $C \in \mathbb{R}^d$. If $C=0$, the support of $\mu$ has to be contained in $\{e_1,\ldots,e_d\}$, which is not possible by assumption.

If $C_i\neq 0$ for some $i \in \{1, \ldots,d\}$, item (a) follows if $C_j=0$ for all $j \neq i$.

If $C_i$ and $C_j$ are nonzero for some $i\neq j$, item (b) follows if $C_\ell=0$ for all $\ell\notin\{i,j\}$. Indeed, by assumption we have $(y^k-e_k)=C_k H(y)$ for $k\in \{i,j\}$ and thus
\begin{align*}
y^i-e_i=\frac{C_i}{C_j}(y^j-e_j).
\end{align*}
Since $y^i,y^j\in\Delta^d$ $\mu$-a.s., we can conclude that $y^i_\ell=y^j_\ell=0$ for all $\ell\notin \{i, j\}$ and hence 
\begin{equation}\label{eqn54}
\frac{y^i_j}{y^j_i}=\frac{1-y^i_i}{y^j_i}=-\frac{C_i}{C_j}=:c
\end{equation}
proving that the conditions of item (b) hold true. In this case $y_i^j\in(0,\frac 1 c \land 1]$ $\mu$-a.s. by (\ref{eqn54}).

Finally, if $C_i\neq0$ for at least three different values of $i$, the same reasoning as for case (b) implies $y_\ell^i=0$ for all $\ell\neq i$ and thus $H=0$ $\mu$-a.s., which is not possible by assumption.

The converse direction is clear.
\end{proof}

\begin{lemma}\label{lem:forma}
The following assertions are equivalent.
\begin{enumerate}
\item 
The matrix $a(x)\in\mathbb S_+^d$ satisfies $a\mathbf 1=0$, $a_{ij} \in \Pol_2(E)$, and
$
a_{ii}=0 \textrm{ on } E \cap \{x_i=0\}
$.
\item 
The matrix $a(x)$ satisfies
$$
a_{ii}(x)=\sum_{i\neq j}\alpha_{ij}x_ix_j\quad\text{and}\quad a_{ij}(x)=-\alpha_{ij}x_ix_j \quad\textrm{ for all }  i\neq j,
$$
for some $\alpha_{ij}=\alpha_{ji} \in \mathbb{R}_+$.
\end{enumerate}
\end{lemma}

\begin{proof}

We start by proving (i) $\Rightarrow$ (ii): By Lemma~\ref{clem1} we already know that for all $i\neq j$ we have
$a_{ij}=-\alpha_{ij}x_ix_j$
for some $\alpha_{ij}\in \R$.
Moreover, as $a\mathbf 1=0$ on $E$, we also have that
$$a_{ii}(x)=-\sum_{j\neq i}a_{ij}(x)=\sum_{i\neq j} \alpha_{ij}x_ix_j$$
 for all $x\in E$. Since $a\in\mathbb S_+^d$ on $E$ and 
 $\alpha_{ij}=4a_{ii}\left((e_i+e_j)/2\right)$
 for all $i\neq j$,
it follows that $\alpha_{ij} \in \mathbb{R}_+$, which finishes the proof of the first direction.
Concerning (ii) $\Rightarrow$ (i), the only condition which is not obvious is positive semidefiniteness of $a$ on $E$, which follows exactly as in the proof of Proposition 6.6 in~\cite{Filipovic/Larsson:2016}.
\end{proof}

\begin{proof}[Proof of Theorem~\ref{th:mainsimplex}]
As $\Gcal$ is polynomial  and its martingale problem is well-posed, the conditions of Lemma~\ref{thm6} are satisfied.
As in Remark \ref{rem3}, we can assume without loss of generality that $\mu$ is compactly supported and all its moments of order greater or equal two are thus finite. 
Analogously to (\ref{eqn7}) we set then
\begin{equation}\label{eqn44}
p_\bn:=\int\gamma^\bn(\fdot,y)\mu(\de y)\quad\text{and}\quad
r_\bn:=\int\xi^\bn\nu(\fdot,\de\xi)=\lambda p_\bn
\end{equation}
for all $|\bn|\geq2$. Note that $p_\bn\in\Pol_{|\bn|}(E)$ for all $|\bn|\geq2$ by the integrability conditions on $\mu$. By condition (iv) of Lemma~\ref{thm6} we also have that $r_\bn\in\Pol_{|\bn|}(E)$ for all $|\bn|\geq3$. By Remark \ref{rem10} we know that 
$$\lambda(x) = \frac{\int |\xi|^4 \nu(x,\de\xi)}{\int |\gamma(x,y)|^4 \mu(\de y)}\mathds1_{\{\int |\gamma(x,y)|^4 \mu(\de y)\neq0\}},$$
and hence condition $\nu\big(x,(E-x)^c\big)=0$ implies that $\mu$ can be chosen supported on $(\Delta^d)^d$ and such that
$$\gamma(x,y)=\sum_{i=1}^d (y^i-e_i) x_i\quad\mu\text{-a.s.}$$
By definition of affine jump sizes, the measure $\mu$ has to be square-integrable.

Concerning the statement on the drift this is a consequence of Lemma~\ref{thm6}. Indeed (iv) yields the affine (and thus linear) form of the drift, (ii) leads to 
\begin{equation}\label{reqn1}
\sum_{j\neq i} \left(B_{ij}x_j -\lambda(x)\int y^j_ix_j \mu({\de y}) \right)\geq 0, \quad x \in E \cap \{x_i=0\}
\end{equation}
and finally $B^{\top} \mathbf{1}=0$ is a consequence of (iii), namely $b^\top\mathbf 1=0$ for all $x\in E$. Since condition \eqref{reqn1}  yields $\sum_{j\neq i} B_{ij} x_j\geq0$, choosing $x=e_j $ we get  $B_{ij}\geq0$ for $j\neq i$ and $B_{ii}=-\sum_{j\neq i}B_{ji}$ for all $i\in\{1,\ldots,d\}$.
We will now consider four complementary assumptions, which will lead to Type~0 to~3.

Assume that $\nu(x,\de\xi)=0$.
Then by Lemma~\ref{thm6} we can apply Lemma \ref{lem:forma} to conclude that $a$ satisfies~\eqref{eq:forma}. This proves that in this case $\Gcal$ is an operator of Type~0.

Assume now that $\nu(x,\de\xi)\neq0$ and $\lambda$ can be chosen to be constant. We can then without loss of generality set $\lambda=1$. Moreover, since in this case $r_{e_i+e_j}\in\Pol_2(E)$ for all $i,j\in\{1,\ldots,d\}$, condition (iv) of Lemma~\ref{thm6} implies that the entries diffusion matrix $a_{ij}\in\Pol_2(E)$.
We can thus conclude as before that $a$ can be chosen to be of form~\eqref{eq:forma}. Finally, condition \eqref{reqn1} can be rewritten as $\sum_{j\neq i} \big(B_{ij}-\int y^j_i \mu({\de y}) \big)x_j\geq 0$ for all $x \in E \cap \{x_i=0\}$, which yields $B_{ij}-\int y^j_i \mu({\de y})$ for all $i\neq j$. As a result, $\Gcal$ is an operator of Type~1.

Assume now that $\nu(x,\de\xi)\neq0$ and $\lambda$ cannot be chosen to be constant. We already know that $\lambda=\frac{p(x)}{q(x)}$ for some $p,q \in \Pol(E)$. 
Supposing without loss of generality that $p(x)$ and $q(x)$ are coprime polynomials, we necessarily have due to Assumption~\ref{A} that $q(x)$ is a divisor of $\gamma_i(x,y)^3$ for each $i \in \{1,\ldots,d\}$ and $\mu$-a.e. $y \in (\Delta^d)^d$. 
Since $\gamma_i(\fdot,y)\in\Pol_1(E)$ $\mu$-a.s., this in turn implies that $\gamma(x,y)=H(y)P_1(x)$ $\mu$-a.s. with a measurable function $H:(\Delta^d )^d \to \mathbb{R}^d$ and $P_1 \in \Pol_1(E)$.

Choose now $i\in\{1,\ldots,d\}$ such that $\mu\big(H_i(y)\neq0\big)>0$. By equation (\ref{eqn44}) we have that
$$\lambda(x)=\frac{r_{3e_i}(x)}{p_{3e_i}(x)}=\frac{\overline r_{3e_i}(x)}{\big(P_1(x)\big)^3}$$
for some $\overline r_{3e_i}\in\Pol_3(E)$, for all $x\in E\setminus\{P_1=0\}$. Since in this case $\nu(x,\de \xi)=0$ for all $x\in E\cap\{P_1=0\}$, we are free to choose $\lambda(x)=0$ on this set.
By Lemma~\ref{thm6} we also know that $r_{2e_i}$ has to be a bounded function on $E$. Noting that for all $x\in E\setminus\{P_1=0\}$
$$r_{2e_i}=\frac{\overline r_{3e_i}(x)}{P_1(x)}\int H^2_i(y) \mu(\de y),$$
we see that this condition holds true if and only if $\overline r_{3e_i}(x)=0$ for all $x\in \{P_1=0\}$.  Since we know by Lemma \ref{lem:formgamma} that $P_1(x)=-cx_i+x_j$ for some $c\geq0$, by Lemma \ref{clem1} we thus have that
$$\lambda(x)=\frac{q_2(x)}{\big(P_1(x)\big)^2}\mathds1_{\{P_1\neq0\}}$$
for some $q_2\in\Pol_2(E)$. This in particular implies that $r_{e_k+e_\ell}\in\Pol_2(E\setminus\{P_1=0\})$ and hence, by condition (iv) of Lemma~\ref{thm6}, $a_{k\ell}\in\Pol_2(E\setminus\{P_1=0\})$ for all $k,\ell\in\{1,\ldots,d\}$. By the same condition we also have that for all $x\in E\cap\{P_1=0\}$
\begin{equation}\label{eqn55}
a_{k\ell}(x)=a_{k\ell}(x)+r_{e_k+e_\ell}(x)=\lim_{z \to x}a_{k\ell}(z)+q_2(x)\int H_k(y)H_\ell(y) \mu(\de y),
\end{equation}
and thus in particular, by positive semidefiniteness of $a(x)$ and condition (ii) of Lemma~\ref{thm6},
\begin{equation}\label{eqn58}
a_{ii}(x)=\lim_{z\to x}a_{ii}(z)=q_2(x)=0
\end{equation}
for all $x\in E\cap\{P_1=0\}\cap\{x_i=0\}$. Setting $a_{k\ell}^c\in\Pol_2(E)$ be such that $a_{k\ell}^c|_{E\setminus\{P_1=0\}}=a_{k\ell}|_{E\setminus\{P_1=0\}}$, we obtain that $a^c:=(a_{k\ell}^c)_{k\ell}$ satisfies the conditions of Lemma~\ref{lem:forma} and thus $a^c$ can be chosen to be of the form (\ref{eq:forma}).
By (\ref{eqn55}) we can then conclude that
$$a(x)=a^c(x)+q_2(x)\mathds1_{\{P_1=0\}}\int H(y)H(y)^\top \mu(\de y).$$
By Lemma \ref{lem:formgamma}, we know that there are only two complementary choices of $H$ and $P_1$.

The first choice is $H(y)=(y^i-e_i)$ and $P_1(x)=x_i$ for some fixed $i\in\{1,\ldots,d\}$. Then by (\ref{eqn58}) and Lemma \ref{clem1} we have $q_2(x)=q_1(x)x_i$ for some $q_1\in\Pol_1(E)$. Moreover, using that $q_1(x)=\sum_{j=1}^d q_1(e_j)x_j$,  condition \eqref{reqn1} can be rewritten as 
$$\sum_{j\neq k} \bigg(B_{kj} -{q_1(e_j)}\mathds1_{\{ x_i\neq 0\}}\int y^i_k \mu({\de y})\bigg)x_j \geq 0, \quad x \in E \cap \{x_k=0\}$$
for all $k\in\{1,\ldots,d\}$, which yields \eqref{eq:simplex 2 bdry} for all $k\neq i$ and $j\neq k$.
As a result,  $\Gcal$ is an operator of Type~2.

The second choice of $H$ and $P_1$ is $H(y)=y_i^j(e_i-e_j)$ and $P_1(x)=-cx_i+x_j$ for some $i,j\in\{1,\ldots,d\}$ such that $i\neq j$, where $y_i^j\in(0,\frac 1 c\land 1]$. Then by (\ref{eqn58}) and Lemma \ref{clem1} we have $q_2(x)=\sum_{k=1}^dq_{ik}x_ix_k+q_{jk}x_jx_k$ for some $q_{k\ell}\in\R$.  Since condition \eqref{reqn1} coincides with condition \eqref{eq:driftpos}, we can conclude that $\Gcal$ is an operator of Type~3.
\end{proof}

\end{appendices}

\bibliographystyle{plainnat}
\bibliography{160804bibl}

\end{document}